\documentclass[a4paper,12pt]{article}

\newcommand{\calZ}{\mathcal{Z}}
\newcommand{\calO}{\mathcal{O}}
\def\softO{\ensuremath{{\mathcal{O}}{\,\tilde{ }\,}}}
\def\bigO{\ensuremath{{\mathcal{O}}}}

\newcommand{\calD}{\mathcal{D}}
\newcommand{\calV}{\mathcal{V}}

\newcommand{\calG}{\mathcal{G}}

\def\NN{{\mathbb{N}}}
\def\ZZ{{\mathbb{Z}}}
\def\QQ{{\mathbb{Q}}}
\def\RR{{\mathbb{R}}}
\def\CC{{\mathbb{C}}}

\def\fall{\forall\,}

\def\minrank{{r(A)}}
\usepackage{varwidth,url}     
\def\MM{{\mathbb{M}}}
\def\Id{{\mathbb{I}}}

\usepackage{amsmath,pgfplots,caption}
\usepackage{amssymb}
\usepackage{amsfonts}
\usepackage{graphicx,epstopdf}
\usepackage{mathrsfs}
\usepackage{subfigure}
\usepackage{comment}

\usepackage{ulem} 

\newtheorem{theorem}{Theorem}
\newtheorem{notation}[theorem]{Notation}   
\newtheorem{proposition}[theorem]{Proposition}   
\newtheorem{lemma}[theorem]{Lemma}   
\newtheorem{corollary}[theorem]{Corollary}

\newtheorem{problem}{Problem}   

\newtheorem{example}[theorem]{Example}   
\newtheorem{conjecture}[theorem]{Conjecture}

\newenvironment{proof}[1][]{\noindent {\bf Proof #1:\;}}{\hfill $\Box$}

\def\QQ{{\mathbb{Q}}} 
\def\RR{{\mathbb{R}}} 
\def\CC{{\mathbb{C}}} 

\def\scC{{\mathscr{C}}}
\def\scS{{\mathscr{S}}}
\def\sfP{{\mathsf{P}}}
\def\sfQ{{\mathsf{Q}}}

\def\sing{{\rm sing}\:}
\def\crit{{\rm crit}\:}
\def\reg{{\rm reg}\:}
\def\spec{\mathscr{S}}
\def\cc{\mathcal{C}}

\newcommand{\zeroset}[1]{{Z(#1)}}
\newcommand{\ideal}[1]{{I(#1)}}

\def\setV{{\mathcal{V}}}

\def\setZ{{\mathcal{Z}}}

\def\zarA{{\mathscr{A}}}
\def\zarfiber{{\mathscr{T}}}

\def\zarM{{\mathscr{M}}}
\def\zarO{{\mathscr{O}}}
\def\zarS{{\mathscr{S}}}
\def\zarC{{\mathscr{C}}}

\def\sfH{{\mathsf{H}}} 

\def\jac{{D}} 

\def\rank{{\rm rank}\,}
\def\deg{{\rm deg}}

\def\GL{{\mathrm{GL}}} 

\def\fiber{{t}}  

\def\softO{\ensuremath{{O}{\,\tilde{ }\,}}}

\def\mymid{{\,\,:\,\,}}

\newcommand{\new}[1]{{{#1}}}
\newcommand{\verynew}[1]{{{#1}}}

\title{Exact algorithms for linear matrix inequalities}

\author{
Didier Henrion\thanks{LAAS-CNRS, Universit\'e de Toulouse, CNRS, Toulouse, France;
Faculty of Electrical Engineering, Czech Technical University in Prague, Czech Republic.} \and 
Simone Naldi\thanks{LAAS-CNRS, Universit\'e de Toulouse, CNRS, Toulouse, France.} \and
Mohab Safey El Din\thanks{Sorbonne Universit\'es, UPMC Univ Paris 06, CNRS, INRIA Paris Center, LIP6, Equipe PolSys, F-75005, Paris, France.}}

\date{\today}

\begin{document}

\maketitle

\begin{abstract}
\noindent
Let $A(x) = A_0+x_1A_1+\cdots+x_nA_n$ be a linear matrix, or pencil, generated by given symmetric
matrices $A_0,A_1,\ldots,A_n$ of size $m$ with rational entries. The set of real vectors $x$
such that the pencil is positive semidefinite is a convex semi-algebraic set called spectrahedron,
described by a linear matrix inequality (LMI). We design an exact algorithm that, up to genericity
assumptions on the input matrices, computes an exact algebraic representation of at least one point in
the spectrahedron, or decides that it is empty. The algorithm does not assume the existence
of an interior point, and the computed point minimizes the rank of the pencil on the spectrahedron.
The degree $d$ of the algebraic representation of the point coincides experimentally with the algebraic
degree of a generic semidefinite program associated to the pencil.
We provide explicit bounds for the complexity of our algorithm, proving that the maximum number
of arithmetic operations that are performed
is essentially quadratic in a multilinear B\'ezout bound of $d$. When $m$ (resp. $n$) is fixed,
such a bound, and hence the complexity, is polynomial in $n$ (resp. $m$). 
We conclude by providing results of experiments showing practical improvements with respect to
state-of-the-art computer algebra algorithms.
\end{abstract}


\section{Introduction} \label{sec1}
Let ${\mathbb{S}}_m(\QQ)$ be the vector space of $m \times m$ symmetric matrices with
entries in $\QQ$. Let $A_0,A_1,\ldots,A_n \in {\mathbb{S}}_m(\QQ)$ and
$A(x)=A_0+x_1A_1+ \cdots+x_nA_n$ be the associated {\it linear matrix},
$x=(x_1, \ldots, x_n) \in \RR^n$.
We also denote the tuple
$(A_0, A_1, \ldots, A_n)$ by $A \in {\mathbb{S}}_m^{n+1}(\QQ)$.
For $x \in \RR^n$, $A(x)$ is symmetric, with real entries, and hence its eigenvalues are real numbers. 

The central object of this paper is the subset of $x \in \RR^n$ such that the eigenvalues
of $A(x)$ are all nonnegative, that is the {\it spectrahedron}
\[
\spec = \{x \in \RR^n \mymid A(x) \succeq 0\}.
\]
Here $\succeq 0$ means ``positive semidefinite'' and $A(x) \succeq 0$ is called
a {\it linear matrix inequality} (LMI). 
The set $\spec$ is convex closed basic semi-algebraic.
This paper addresses the following decision problem for the spectrahedron $\spec$: 
\begin{problem}[\new{Feasibility of semidefinite programming}]
\label{mainprob}
Compute an exact algebraic representation of at least one
point in $\spec$, or decide that $\spec$ is empty.
\end{problem}

We present a probabilistic \new{exact} algorithm for solving Problem
\eqref{mainprob}.  \new{The algorithm depends on some assumptions on
  input data that are specified later.}  If $\spec$ is not empty, the
expected output is a {\it rational parametrization} (see
e.g. \cite{rur}) of a finite set $\calZ \subset \CC^n$ meeting $\spec$
in at least one point $x^*$.  This is given by a vector
$(q_0,\ldots,q_{n+1}) \subset \ZZ[t]^{n+2}$ and a linear form
$\lambda=\lambda_1 x_1+\cdots+\lambda_nx_n\in \QQ[t]$ such that
$\deg(q_{n+1})=\sharp\calZ$, $\deg(q_i)<\deg(q_{n+1})$ for $0\leq i \leq n$,
${\rm gcd}(q_{n+1}, q_0)=1$ and $\calZ$ coincides with the set
\begin{equation}
\label{rationalparametrization}
\{(x_1, \ldots, x_n)\in \CC^n \mid t=\lambda_1x_1+\cdots+\lambda_n x_n, q_{n+1}(t)=0, x_i=q_i(t)/q_0(t)\}
\end{equation}
\new{A few remarks on this representation are in order. Usually, $q_0$
  is taken as $\frac{\partial q_{n+1}}{\partial t}$ for a better
  control on the size of the coefficients \cite[Theorem 1]{DS04} --
  see also the introductory discussion of that theorem in
  \cite{DS04}.}  \verynew{More precisely, if $D$ bounds the degrees of
  a finite family of polynomials in $\ZZ[x_1, \ldots, x_n]$ defining
  $\calZ$ and $h$ bounds the bit size of their coefficients and those
  of $\lambda$, then the coefficients of a rational parametrization
  encoding $\calZ$ have bit size bounded by $hD^n$.}

\verynew{This is to be compared with polynomial parametrizations where
  the rational fractions $q_i/q_0$ are replaced by polynomials $p_i$;
  they are obtained by inverting $q_0$ w.r.t. $q_{n+1}$ using the extended
  Euclidean algorithm. That leads to polynomials $p_i$ with bit size
  bounded by $hD^{2n}$.}

\verynew{In order to compute such representations, the usual and
  efficient strategy is to compute first the image of such
  representations in a prime field and next use a Newton-Hensel
  lifting to recover the integers. According to \cite[Lemma
  4]{GiLeSa01} and the above bounds, the cost of lifting integers is
  $\log$-linear in the output size. Since in the case of polynomial
  parametrizations, the output size may be $D^n$ times larger than in
  the case of rational parametrizations, rational parametrizations are
  easier to compute. In addition, observe {f}rom \cite[Lemma 3.4 and
  Theorem 3.12]{MS16} that isolating boxes for the real points in
  $\calZ$ from rational or polynomial representations have the same
  bit complexity (e.g. cubic in the degree of $q$ and $\log$-linear in
  the maximum bit size of the coefficients in the parametrization).}

\medskip

As an outcome of designing our algorithm, we also compute the minimum
rank attained by the pencil on the spectrahedron.  Moreover, since the
points in $\calZ$ are in one-to-one correspondence with the roots of
$q_{n+1}$, from this representation the coordinates of
the feasible point $x^* \in \spec$ can be computed with arbitrary
precision by isolating the corresponding solution $t^*$ of the
univariate equation $q_{n+1}(t)=0$. If $\spec$ is empty, the expected
output is the empty list.

\subsection{Motivations} \label{sec1:ssec1}

Semidefinite programming models a large number of problems in the applications \cite{RaGo,boydvan,bpt13}.
This includes one of the most important questions in computational
algebraic geometry, that is the general polynomial optimization problem. Indeed, Lasserre
\cite{lasserre01} proved that the problem of minimizing
a polynomial function over a semi-algebraic set can be relaxed to a sequence of primal-dual
semidefinite programs called LMI relaxations, and that under mild assumptions the sequence of solutions converge
to the original minimum. Generically, solving a non-convex polynomial optimization problem amounts to
solving a finite-dimensional convex semidefinite programming problem \cite{nie14}.
Numerical algorithms following this approach are available and,
typically, guarantees of their convergence are related to the feasibility (or
strict feasibility) of the LMI relaxations. It is, in general, a challenge to
obtain exact algorithms for deciding whether the feasible set of a semidefinite programming (SDP)
problem
\begin{equation}
\label{semdefprog}
\min_{x \in \RR^n} \sum_{i=1}^n c_ix_i \qquad \text{s.t.} \,\, A(x) \succeq 0
\end{equation}
is empty or not. The feasible set of the SDP \eqref{semdefprog} is defined by an LMI
and hence it is a spectrahedron.
Problem \eqref{mainprob} amounts to solving the feasibility problem for semidefinite
programming, in exact arithmetic: given a $\QQ-$definable semidefinite program as in \eqref{semdefprog}
(that is, we suppose that the coefficients of $A(x)$ have rational entries), decide
whether the feasible set $\spec = \{x \in \RR^n \mymid A(x) \succeq 0\}$ is empty or not,
and compute exactly at least one feasible point.
We would like to emphasize the fact that
we do not assume the existence of an interior point in $\spec$. Quite the opposite,
we are especially interested in degenerate cases for which the maximal rank
achieved by the pencil $A(x)$ in $\spec$ is small.

This work is a first step towards an exact approach to semidefinite programming. 
In particular, a natural perspective of this work is to design exact algorithms
for deciding whether the minimum in \eqref{semdefprog} is attained or not, and
for computing such a minimum in the affirmative case. While the number of iterations performed
by the ellipsoid algorithm \cite{ellipsoid} to compute the approximation of a solution
of \eqref{semdefprog} is polynomial in the number of variables, once the accuracy
is fixed, no analogous results for exact algorithms are available. Moreover,
since the intrinsic complexity of the optimization problem \eqref{semdefprog}
is related to its algebraic degree $\delta$ as computed in \cite{stu,SDPformula}, the paramount
goal is to design algorithms whose runtime is polynomial in $\delta$. The algorithm
of this paper shows experimentally such an optimal behavior with respect to $\delta$.

We finally recall that solving LMIs is a basic subroutine of
computer algorithms in systems control and optimization, especially in linear systems
robust control \cite{befb94,henNotes}, but also for the analysis or synthesis of nonlinear
dynamical systems \cite{sophie}, or in nonlinear optimal control with polynomial data
\cite{optimContr,claeys}.

\subsection{Contribution and outline} \label{sec1:ssec3}

We design a computer algebra algorithm for solving the feasibility
problem of semidefinite programming, that is Problem \eqref{mainprob},
in exact arithmetic.  Let us clarify that we do not claim that an
exact algorithm can be competitive with a numerical algorithm in terms
of admissible size of input problems: indeed, SDP solvers based on
interior-point methods \cite{bn01,nestnem} can nowadays handle inputs
with a high number of variables that are out of reach for our
algorithms. Our contribution can be summarized as follows:
\begin{enumerate}
\item we show that the geometry of spectrahedra understood as
  semi-algebraic sets with determinantal structure can be exploited to
  design dedicated computer algebra algorithms;
\item we give explicit complexity and output-degree upper bounds for
  computer algebra algorithms solving exactly the feasibility problem
  of semidefinite programming; \new{our algorithm is probabilistic and
    works under assumptions on the input, which are generically
    satisfied};
\item we provide results of practical experiments showing the gain in
  terms of computational timings of our contribution with respect to
  the state of the art in computer algebra;
\item remarkably, our algorithm does not assume that the input
  spectrahedron $\spec = \{x \in \RR^n \mymid A(x) \succeq 0\}$ is
  full-dimensional, and hence it can tackle also examples with empty
  interior.
\end{enumerate}

The main idea is to exploit the relation between the geometry of
spectrahedra, and that of the determinantal varieties associated to
the input symmetric pencil $A(x)$. Let us introduce, for
$r = 0, \ldots, m-1$, the algebraic sets
\[
\calD_r = \{x \in \CC^n \mymid \rank\,A(x)\leq r\}.
\]
These define a nested sequence
$\calD_0 \subset \calD_1 \subset \cdots \subset \calD_{m-1}$. \verynew{The dimension
of $\calD_r$ for generic linear matrices $A$ is known, and equals $n-\binom{m-r+1}{2}$
(see Lemma \ref{dimdetvarsym}).} The
Euclidean boundary $\partial \spec$ of $\spec$ is included in the real
trace of the last algebraic set of the sequence:
$\partial \spec \subset \calD_{m-1} \cap \RR^n.$ In particular, for
$x \in \partial \spec$, the matrix $A(x)$ is singular and one could
ask which elements of the real nested sequence
$\calD_0 \cap \RR^n \subset \cdots \subset \calD_{m-1} \cap \RR^n$
intersect $\partial \spec$.

\begin{notation} \label{minrank}
If $\spec = \{x \in \RR^n \mymid A(x) \succeq 0\}$ is not empty, we define the integer
$\minrank = \min \left\{\rank A(x) \mymid x \in \spec\right\}.$
\end{notation}

When $\spec$ is not empty, $\minrank$ equals the minimum integer $r$
such that $\calD_r \cap \RR^n$ intersects $\spec$.  We present our
first main result, which states that $\spec$ contains at least one of
the connected components of the real algebraic set
$\calD_\minrank \cap \RR^n$.  We denote by
${\mathbb{S}}_m^{n+1}(\QQ) = {\mathbb{S}}_m(\QQ) \times \cdots \times {\mathbb{S}}_m(\QQ)$ the
$(n+1)-$fold Cartesian product of ${\mathbb{S}}_m(\QQ)$.

\begin{theorem}[Smallest rank on a spectrahedron] \label{theo:spectra}
Suppose that $\spec \neq \emptyset$.
Let $\cc$ be a connected component of $\calD_\minrank \cap \RR^n$ such that
$\cc \cap \spec \neq \emptyset$. Then $\cc \subset \spec$ and hence $\cc \subset
(\calD_\minrank \setminus \calD_{\minrank-1}) \cap \RR^n$.
\end{theorem}

We give a proof of Theorem \ref{theo:spectra} in Section \ref{sec2}.
From Theorem \ref{theo:spectra}, we deduce the following mutually
exclusive conditions on the input symmetric linear pencil $A$: either
$\spec = \emptyset$, or $\spec$ contains one connected component $\cc$
of $\calD_\minrank \cap \RR^n$.  Consequently, an exact algorithm
whose output is one point in the component
$\cc \subset \spec \cap \calD_\minrank$ would be sufficient for our
goal.  Motivated by this fact, we design in Section \ref{sec3:ssec2}
an exact algorithm computing one point in each connected component of
$\calD_r \cap \RR^n$, for $r \in \{0,\ldots,m-1\}$. 

The strategy to compute sample points in $\calD_r \cap \RR^n$ is to
build an algebraic set $\setV_r \subset \CC^{n+m(m-r)}$ whose
projection on the first $n$ variables is contained in $\calD_r$. This
set is defined by the incidence bilinear relation $ A(x)Y(y)=0 $ where
$Y(y)$ is a full-rank $m \times (m-r)$ linear matrix whose columns
generate the kernel of $A(x)$ ({\it cf.} Section
\ref{sec3:ssec1}). Unlike $\calD_r$, the incidence variety $\setV_r$,
up to genericity conditions on the matrices $A_0,A_1,\ldots,A_n$,
turns to be generically smooth and equidimensional.  The next theorem
presents a complexity result for an exact algorithm solving Problem
\eqref{mainprob} under these genericity assumptions.

\begin{theorem}[Exact algorithm for LMI] \label{maintheo2} Suppose
  that for $0 \leq r \leq m-1$, the incidence variety $\calV_r$ is
  smooth and equidimensional and that its defining polynomial system
  generates a radical ideal. Suppose that $\calD_r$ has the expected
  dimension $n-\binom{m-r+1}{2}$. There is a
  probabilistic algorithm that takes $A$ as input and returns:
\begin{enumerate}
\item either the empty list, if and only if $\spec = \emptyset$, or
\item a vector $x^*$ such that $A(x^*)=0$, if and only if the linear
  system $A(x)=0$ has a solution, or
\item a rational parametrization
  $q = (q_0,\ldots,q_{n+1}) \subset \ZZ[t]$ such that there exists
  $t^* \in \RR$ with $q_{n+1}(t^*)=0$ and:
  \begin{itemize}
  \item $A(q_1(t^*)/q_0(t^*),\ldots,q_n(t^*)/q_0(t^*)) \succeq 0$ and
  \item $\rank A(q_1(t^*)/q_0(t^*),\ldots,q_n(t^*)/q_0(t^*)) = \minrank$.
  \end{itemize}
\end{enumerate}
The number of arithmetic operations performed are in 
\[
\softO\left( n\,\binom{\frac{m^2+m}{2}+n}{n}^6\,\sum_{r \leq m-1} \binom{m}{r}\,(n+(m-r)(m+3r))^{7}\right).
\]
If $\spec \neq \emptyset$, the degree of $q$ is in
$
\bigO\left(\binom{\frac{m^2+m}{2}+n}{n}^3\right).
$
\end{theorem}

\verynew{An important aspect of our contribution can be read from the complexity and degree
bounds in Theorem \ref{maintheo2}: indeed, remark that when $m$ is fixed, both the output degree
and the complexity of the algorithm are polynomial functions of $n$. Viceversa, by the constraint
$n\geq\binom{m-r+1}{2}$ (given by Lemma \ref{dimdetvarsym}), one also easily deduces that when $n$ is fixed, the complexity
is polynomial in $m$.}

The algorithm of Theorem \ref{maintheo2} is described in
Section~\ref{sec3}.  Its probabilistic nature comes from random
changes of variables performed during the procedure, allowing to put
the sets $\calD_r$ in generic position. We prove that for generic
choices of parameters the output of the algorithm is correct.

A complexity analysis is performed in Section \ref{sec4}. Bounds in
Theorem \ref{maintheo2} are explicit expressions involving $m$ and $n$.
These are computed by exploiting the multilinearity
of intermediate polynomial systems generated during the procedure, and
are not sharp in general. By experiments on randomly generated
symmetric pencils, reported in Section \ref{sec5}, we observe that the
output degree coincides with the algebraic degree of generic
semidefinite programs, that is with data given in
\cite[Table\,2]{stu}: this evidences the optimality of our approach.
We did not succeed in proving exact formulas for such degrees. 

\subsection{Related works} \label{sec1:ssec2} \verynew{On input $A$,
  $\spec$ can be defined by $m$ polynomial inequalities in
  $\QQ[x_1, \ldots, x_n]$ of degree $\leq m$ (see
  e.g. \cite{PowersWor}). As far as we know, the state-of-the-art for
  designing algorithms deciding the emptiness of $\spec$ consists only
  of algorithms for deciding the emptiness of general semi-algebraic
  sets; our contribution being the first attempt to exploit structural
  properties of the problem, e.g. through the smallest rank property
  (Theorem \ref{theo:spectra}).}

A first algorithmic solution to deciding the emptiness of general
semi-algebraic sets is given by Cylindrical Algebraic Decomposition
algorithm \cite{c-qe-1975}; however its runtime is doubly exponential
in the number $n$ of variables.  \verynew{The first singly exponential
  algorithm is given in \cite{GV88}, and has led to a series of works
  (see e.g. \cite{reneg, heiRoSol2}) culminating with algorithms
  designed in \cite{BaPoRo1996} based on the so-called critical points
  method. This method is based on the general idea which consists in
  computing minimizers/maximizers of a well-chosen function reaching
  its extrema in each connected component of the set under
  study. Applying \cite{BaPoRo1996} to problem \eqref{mainprob}
  requires $m^{O(n)}$ bit operations. Note that our technique for
  dealing with sets $\setV_r$ is based on the idea underlying the
  critical point method. Also, in the arithmetic complexity model, our
  complexity estimates are more precise (the complexity constant in
  the exponent is known) and better.}  \verynew{This technique is also
  related to algorithms based on polar varieties for grabbing sample
  points in semi-algebraic sets; see for example
  \cite{BGHSS,BGHM1,SaSc03,SaSc04} and its application to polynomial
  optimization \cite{greuet2}. }

\verynew{To get a purely algebraic certificate of emptiness for $\spec$, one
could use the classical approach by Positivstellensatz
\cite{laurent,putinar,schmudgen}.  As a snake biting its tail, this
would lead to a family, or {hierarchy}, of semidefinite programs
\cite{lasserre01}. Bounds for the degree of Positivstellensatz
certificates are exponential in the number of variables and have been
computed in \cite{mark5} for Schmudgen's, and in \cite{mark4} for
Putinar's formulation.  In the recent remarkable result
\cite{lombPerRoy}, a uniform $5-$fold exponential bound for the degree
of the Hilbert 17th problem is provided. In \cite{mark3}, an emptiness
certificate dedicated to the spectrahedral case, by means of special
quadratic modules associated to these sets, is obtained.
}

\verynew{All the above algorithms do not exploit the particular
  structure of spectrahedra understood as determinantal semi-algebraic
  sets.  In \cite{khapor}, the authors showed that deciding emptiness
  of $\spec$ can be done \new{in time $\bigO(nm^4)+m^{\bigO(\min(n,m^2))}$}, that is in
  \new{polynomial time in $m$ (resp. linear time in $n$) if $n$
    (resp. $m$) is fixed}.  The main drawback of this algorithm is
  that it is based on general procedures for quantifier elimination,
  and hence it does not lead to efficient practical implementations.
  Note also that the complexity constant in the exponent is still
  unknown.  }

\verynew{Also, in \cite{ratLMI}, a version of \cite{SaZhi10} dedicated
  to spectrahedra exploiting some of their structural properties,
  decides whether a linear matrix inequality $A(x) \succeq 0$ has a
  rational solution, that is whether $\spec$ contains a point with
  coordinates in $\QQ$.  Remark that such an algorithm is not
  sufficient to solve our problem, since, in some degenerate but
  interesting cases, $\spec$ is not empty but does not contain
  rational points: in Section \ref{sec5:ssec2} we will illustrate the
  application of our algorithm to one of these examples.}

\verynew{As suggested by the smallest rank property, determinantal
  structures play an important role in our algorithm. This structure
  has been recently exploited in \cite{fauspasafey} and \cite{FSS12}
  for the fast computation of Gr\"obner bases of zero-dimensional
  determinantal ideals and computing zero-dimensional critical loci of
  maps restricted to varieties in the generic case.}

\verynew{ Exploiting determinantal structures for determinantal
  situations remained challenging for a long time.  In \cite{HNS2014}
  we designed a dedicated algorithm for computing sample points in the
  real solution set of the determinant of a square linear matrix. This
  has been extended in \cite{HNS2015b} to real algebraic sets defined
  by rank constraints on a linear matrix. Observe that this problem
  looks similar to the ones we consider thanks to the smallest rank
  property. As in this paper, the traditional strategy consists in
  studying incidence varieties for which smoothness and regularity
  properties are proved under some {\it genericity assumptions} on the
  input linear matrix.}

\verynew{Hence, in the case of symmetric matrices, these results
  cannot be used anymore. Because of the structure of the matrix, the
  system defining the incidence variety involves too many equations;
  some of them being redundant. Hence, these redundancies need to be
  eliminated to characterize critical points on incidence varieties in
  a convenient way. In the case of Hankel matrices, the special
  structure of their kernel provides an efficient way to do that. This
  case study is done in \cite{HNS2015a}. Yet, the problem of
  eliminating these redundancies remained unsolved in the general
  symmetric case and this is what we do in Section~\ref{sec3} which is
  the starting point of the design of our dedicated algorithm.}

\subsection{Basic notation}

We refer to \cite{BaPoRo06,CLO,harris1992algebraic,Eisenbud95} for the
algebraic-geometric background of this paper.  We recall below some
basic definitions and notation.  We denote by $\MM_{p,q}(\QQ)$ the
space of $p \times q$ rational matrices, and $\GL_n(\CC)$ the set of
$n \times n$ non-singular matrices. The transpose of
$M \in \MM_{p,q}(\QQ)$ is $M^T$.  The cardinality of a finite set $T$
(resp. the number of entries of a vector $v$) are denoted by
$\sharp T$ (resp. $\sharp v$).

Let $x=(x_1,\ldots,x_n)$. A vector $f=(f_1,\ldots,f_s) \subset \QQ[x]$ is a polynomial system, $\langle f \rangle \subset \QQ[x]$
its ideal and $\zeroset{\langle f \rangle}=\{x \in \CC^n \mymid f_i(x)=0, i=1,\ldots,s\}$ the associated algebraic set.
Sets $\zeroset{\langle f \rangle}$ define the collection of closed sets of the Zariski topology of $\CC^n$. The intersection of
a Zariski closed and a Zariski open set is called a locally closed set. For $M \in \GL_n(\CC)$ and $\calZ \subset \CC^n$, let
$M^{-1}\calZ=\{x \in \CC^n \mymid M\,x \in \calZ\}$. With $\ideal{S}$ we denote the ideal of polynomials vanshing on $S \subset \CC^n$.

Let $f=(f_1,\ldots,f_s) \subset \QQ[x]$. Its Jacobian matrix is denoted by $\jac f = ({\partial f_i}/{\partial x_j})_{i,j}$.
An algebraic set $\calZ \subset \CC^n$ is irreducible if $\calZ = \calZ_1 \cup \calZ_2$ where $\calZ_1,\calZ_2$
are algebraic sets, implies that either $\calZ=\calZ_1$ or $\calZ=\calZ_2$. Any algebraic
set is the finite union of irreducible algebraic sets, called
its irreducible components. The codimension $c$ of an irreducible algebraic set $\calZ \subset \CC^n$
is the maximum rank of $\jac f$ on $\calZ$, where $\ideal{\calZ} = \langle f
\rangle$. Its dimension is $n-c$. If all the irreducible components of $\calZ$ have the
same dimension, we say that $\calZ$ is equidimensional.
The dimension of an algebraic set $\calZ$ is the maximum of the dimensions of its irreducible components,
and it is denoted by $\dim \calZ$. The degree of an equidimensional
algebraic set $\calZ$ of codimension $c$ is the maximum cardinality of finite intersections
$\calZ \cap {\cal L}$ where ${\cal L}$ is a linear space of dimension $c$. The degree of an algebraic
set is the sum of the degrees of its equidimensional components.

Let $\calZ \subset \CC^n$ be equidimensional of codimension $c$, and let $\ideal{\calZ} =
\left\langle f_1, \ldots, f_s \right\rangle$. The singular locus of $\calZ$,
denoted by $\sing(\calZ)$, is the algebraic set defined by $f=(f_1, \ldots, f_s)$
and by all $c \times c$ minors of $\jac f$. 
If $\sing(\calZ)=\emptyset$ we
say that $\calZ$ is smooth, otherwise singular. The points in
$\sing(\calZ)$ are called singular, while points in $\reg(\calZ) = \calZ \setminus
\sing(\calZ)$ are called regular.
Let $\calZ \subset \CC^n$ be smooth and equidimensional of codimension $c$, and
let $\ideal{\calZ} = \left\langle f_1, \ldots, f_s \right\rangle$.
Let $g : \CC^n \to \CC^m$ be an algebraic map.
The set of critical points of the restriction of $g$ to $\calZ$ is the algebraic set denote by
$\crit(g,\calZ)$ and defined by $f=(f_1, \ldots, f_s)$ and by all $c+m$ minors of the
Jacobian matrix $\jac (f,g).$
The points in $g(\crit(g,\calZ))$ are called critical values,
while points in $\CC^m \setminus g(\crit(g,\calZ))$ are called the regular values, of the restriction
of $g$ to $\calZ$.

\section{The smallest rank on a spectrahedron} \label{sec2}

We prove Theorem \ref{theo:spectra}, which relates the geometry of linear matrix inequalities to the
rank stratification of the defining symmetric pencil. We believe that the statement of this theorem
is known to the community of researchers working on real algebraic geometry and semidefinite
optimization; however, we did not find any explicit reference in the literature.

\begin{proof}[of Theorem \ref{theo:spectra}]
By assumption, the rank of $A(x)$ on $\spec$ is greater or equal than $\minrank$.
We consider the vector function
$
  \begin{array}{lrcc}
  e = (e_1, \ldots, e_m): &  \RR^n & \longrightarrow & \RR^m \\
  \end{array}
$
where $e_1(x) \leq \ldots \leq e_m(x)$ are the ordered eigenvalues
of $A(x)$. Let $\cc$ be a connected component of $\calD_{\minrank} \cap \RR^n$
such that $\cc \cap \spec \neq \emptyset$, and let $x \in \cc \cap \spec$.
One has $\rank A(x) = \minrank$ and 
$
e_1(x) = \ldots = e_{m-\minrank}(x) = 0 < e_{m-\minrank+1}(x) \leq \ldots \leq e_m(x).
$
Suppose {\new{\it ad absurdum}} that there exists $y \in \cc$ such that $y \notin \spec$.
In particular, one eigenvalue of $A(y)$ is strictly negative.

Let $g \colon [0,1] \to \cc$ be a continuous semi-algebraic map such that $g(0)=x$
and $g(1)=y$. This map exists since $\cc$ is a connected component of a real algebraic
set. The image $g([0,1])$ is compact and semi-algebraic. Let
$$
T=\{t \in [0,1] \mymid g(t) \in \spec \} = g^{-1}(g([0,1]) \cap \spec).
$$
Since $g$ is continuous, $T \subset [0,1]$ is closed. So it is a finite union of
closed intervals. Since $0 \in T$ ($g(0)=x \in \spec$) there exists
$t_0 \in [0,1]$ and $N \in \mathbb{N}$ such that $[0,t_0] \in T$ and for all
$p \geq N$, $t_0+\frac{1}{p} \notin T$.
One gets that $g(t_0) = \tilde{x} \in \spec$ and that for all $p \geq N$,
$g(t_0+\frac{1}{p}) = \tilde{x}_p \notin \spec$. By definition, $\tilde{x}, \tilde{x}_p
\in \cc \subset \calD_{\minrank} \cap \RR^n$ for all $p \geq N$, and since $\tilde{x} \in \spec$,
we get $\rank A(\tilde{x}) = \minrank$ and $\rank A(\tilde{x}_p) \leq \minrank$ for all $p \geq N$.
We also get that $\rank A(g(t))=\minrank$ for all $t \in [0, t_0]$.
We finally have $\tilde{x}_p \to \tilde{x}$ when $p \to +\infty$, since $g$ is continuous.
There exists a map
$$
\varphi \colon \Big\{p \in \mathbb{N} \mymid p \geq N \Big\} \to \ZZ
$$
which assigns to $p$ the index of eigenvalue-function among
$e_1, \ldots, e_m$ corresponding to the maximum
strictly negative eigenvalue of $A(\tilde{x}_p)$, if it exists, otherwise
it assigns $0$.
Remark that since $\rank A(\tilde{x}_p) \leq \minrank$ for all
$p$, then $\varphi(p) \leq \minrank$ for all $p$.
In other words, the eigenvalues of $A(\tilde{x}_p)$ satisfy
\begin{align*}
e_1(\tilde{x}_p) \leq \ldots \leq e_{\varphi(p)}(\tilde{x}_p) < 0 = e_{\varphi(p)+1}(\tilde{x}_p) = \ldots = e_{\varphi(p)+m-\minrank}(\tilde{x}_p) \\
0 \leq e_{\varphi(p)+m-\minrank+1}(\tilde{x}_p) \leq \ldots \leq e_m(\tilde{x}_p),
\end{align*}
for $p \geq N$.
Since the sequence $\{\varphi(p)\}_{p \geq N}$ is bounded, up to taking a subsequence, it admits at least
a limit point by the Bolzano-Weierstrass Theorem \cite[Th.\,3.4.8]{bartle}, \new{this
point is an integer, and $j \mapsto \varphi(j)$ is constant for large $j$.}
Suppose that there exists a limit point $\ell > 0$, and let
$\{p_j\}_{j \in \mathbb{N}}$ such that $\varphi(p_j) \rightarrow \ell$ and that for $j \geq N'$,
$j \mapsto \varphi(p_j)$ is constant.
Thus, $0 = e_{\ell+1}(\tilde{x}_{p_j}) = \ldots = e_{\ell+m-\minrank}(\tilde{x}_{p_j})$
for all $j \geq N'$. Since $\tilde{x}_{p_j} \to \tilde{x}$, and since $e_1, \ldots, e_m$
are continuous functions, we obtain that $\ell=0$ is the unique
limit point of $\varphi$, hence $\varphi$ converges to $0$. Hence $\varphi \equiv 0$
for large $p$. \new{This contradicts the fact that $\tilde{x}_p \notin \spec$ for large
$p$.}

We conclude that the set $\cc \setminus \spec$ is empty, that is
$\cc \subset \spec$. By the minimality of the integer $\minrank$ in
$\{\rank A(x) \mymid x \in \spec\}$, one  deduces
that $\cc \subset (\calD_{\minrank} \setminus \calD_{\minrank-1}) \cap \RR^n$.
\end{proof}

\section{Algorithm} \label{sec3}

Our algorithm is called {\sf SolveLMI}, and it is presented in Section \ref{sec3:ssec3}.
Before, we describe in Section \ref{sec3:ssec2} its main subroutine {\sf LowRankSym}, which is of recursive
nature and computes one point per connected component of the real algebraic set
$\calD_r \cap \RR^n$. We start, in the next section, with some preliminaries.

\subsection{Preliminaries} \label{sec3:ssec1}

\subsubsection*{Expected dimension of low rank loci}

We first recall a known fact about the dimension of $\calD_r$, when $A$ is a generic symmetric pencil.

\begin{lemma}
\label{dimdetvarsym}
There exists a non-empty Zariski open subset $\zarA \subset {\mathbb{S}}_{m}^{n+1}(\CC)$
such that, if $A \in \zarA \cap {\mathbb{S}}_{m}^{n+1}(\QQ)$, for all $r=0,\ldots, m-1$, the
set $\calD_r$ is either empty or it has dimension $n-\binom{m-r+1}{2}.$
\end{lemma}

\begin{proof}
The proof is classical and can be found {\it e.g.} in  \cite[Prop.\,3.1]{anjlas}.
%
\end{proof}

\subsubsection*{Incidence varieties}
Let $A(x)$ be a symmetric $m \times m$ linear matrix, and let $0 \leq r \leq m-1$.
Let $y=(y_{i,j})_{1 \leq i \leq m, 1 \leq j \leq m-r}$ be unknowns. Below, we build an algebraic set
whose projection on the $x-$space is contained in $\calD_r$. Let
\[
Y(y)=
\left(
\begin{array}{ccc}
y_{1,1} & \hspace{-0.2cm} \cdots \hspace{-0.2cm} & y_{1,m-r} \\
\vdots & \hspace{-0.2cm} \hspace{-0.2cm}        & \vdots  \\
y_{m,1} & \hspace{-0.2cm} \cdots \hspace{-0.2cm} & y_{m,m-r}
\end{array}
\right),
\]
and let $\iota=\{i_1, \ldots, i_{m-r}\} \subset \{1, \ldots, m\}$, with $\sharp \iota = m-r$.
We denote by $Y_{\iota}$ the $(m-r) \times (m-r)$ sub-matrix of $Y(y)$ obtained by
isolating the rows indexed by $\iota$. There are $\binom{m}{r}$ such matrices. We define
the set
\[
\calV_r(A, \iota) = \{(x, y) \in \CC^n \times \CC^{m(m-r)} \mymid A(x) Y(y) = 0, Y_{\iota} - \Id_{m-r} = 0\}.
\]
We denote by $f(A, \iota)$, or simply by $f$, when there is no ambiguity on $\iota$, the
polynomial system defining $\calV_r(A, \iota)$. \new{We often consider linear changes of variables $x$:} for $M \in \GL_n(\CC)$,
$f(A \circ M, \iota)$ denotes the entries of $A(M\,x) Y(y)$ and $Y_\iota-\Id_{m-r}$, and by $\calV_r(A \circ M, \iota)$
its zero set.
We also denote by $U_\iota \in \MM_{m-r,m}(\QQ)$ the full rank matrix whose entries are in $\{0,1\}$,
and such that $U_\iota Y(y) = Y_\iota$. By simplicity we call $U_\iota$ the {\it boolean matrix} with
multi-index $\iota$.

\new{By definition, the projection of $\calV_r(A,\iota)$ on the first $n$ variables is contained
in $\calD_r$.}
We remark the similarity between the relation $A(x)Y(y)=0$ and the so-called {\it complementarity
condition} for a couple of primal-dual semidefinite programs, see for example \cite[Th.\,3]{stu}.
The difference in our model is that the special size of $Y(y)$ and the affine constraint $Y_\iota=
\Id_{m-r}$ force a rank condition on $Y(y)$ and hence on $A(x)$.

\subsubsection*{Eliminating redundancies}

The system $f(A,\iota)$ contains redundancies induced by polynomial relations between its generators.
These relations can be explicitly eliminated in order to obtain a minimal polynomial system defining
$\calV_r$.

\begin{lemma} \label{lemma:cleaningrelations}
Let $M \in \GL_n(\CC)$. Let $\iota \subset \{1, \ldots, m\}$, with $\sharp\iota=m-r$. Let $A \in {\mathbb{S}}_{m}^{n+1}(\QQ)$,
and $f \in \QQ[x,y]^{m(m-r)+(m-r)^2}$ be the polynomial system defined in Section \ref{sec3:ssec1}. Then we can
explicitly construct a subsystem $f_{red} \subset f$ of length $m(m-r)+\binom{m-r+1}{2}$ such that
$\langle f_{red} \rangle = \langle f \rangle$.
\end{lemma}

\begin{proof}
In order to simplify notations and without loss of generality we suppose $M=\Id_n$ and $\iota = \{1, \ldots, m-r\}$.
We substitute $Y_\iota=\Id_{m-r}$ in $A(x) Y(y)$, and we denote by $g_{i,j}$ the $(i,j)-$th entry of the resulting
matrix. We denote by $f_{red}$ the following system:
\[
\new{f_{red} = \left(g_{i,j} \,\, \text{for} \,\, i \geq j, \,\,Y_\iota-\Id_{m-r}\right).}
\]
We claim that, for $1 \leq i \neq j \leq m-r$, $g_{i,j} \equiv g_{j,i} \mod \langle g_{k,\ell},  k>m-r \rangle$,
which implies that $f_{red}$ verifies the statement. Let $a_{i,j}$ denote the $(i,j)-$th entry of $A(x)$. Let $i<j$
and write
\[
g_{i,j} = a_{i,j} + \sum_{\ell=m-r+1}^m a_{i,\ell} y_{\ell,j}
\,\,\, {\rm and} \,\,\,
g_{j,i} = a_{j,i} + \sum_{\ell=m-r+1}^m a_{j,\ell} y_{\ell,i}.
\]
We deduce that $g_{i,j} - g_{j,i} = \sum_{\ell=m-r+1}^m {a_{i,\ell} y_{\ell,j} - a_{j,\ell} y_{\ell,i}}$
since $A$ is symmetric. Also, modulo the ideal $\left\langle g_{k,\ell}, \, k>m-r \right\rangle$,
and for $\ell \geq m-r+1$, one can explicit $a_{i,\ell}$ and $a_{j,\ell}$, by using polynomial
relations $g_{\ell,i}=0$ and $g_{\ell,j}=0$, as follows:
\begin{align*}
g_{i,j} - g_{j,i}
& \equiv \sum_{\ell=m-r+1}^m \left(-\sum_{t=m-r+1}^{m}a_{\ell,t}y_{t,i}y_{\ell,j}+\sum_{t=m-r+1}^{m}a_{\ell,t}y_{t,j}y_{\ell,i} \right) \equiv \\
& \equiv \sum_{\ell,t=m-r+1}^m a_{\ell,t} \left( -y_{t,i}y_{\ell,j}+y_{t,j}y_{\ell,i} \right) \equiv 0 \,\,\, \mod \left\langle g_{k,\ell}, \, k>m-r \right\rangle.
\end{align*}
The previous congruence concludes the proof.
\end{proof}

We prove below in Proposition \ref{prop:regularity:sym} and in Corollary \ref{cor:completeInt} that,
up to genericity assumptions, the ideal $\langle f \rangle = \langle f_{red} \rangle$ is radical and that
$\sharp f_{red}$ matches the codimension of $\calV_r$. In the next example, we explicitly write
down the redundancies shown in Lemma \ref{lemma:cleaningrelations} for a simple case.
\begin{example}\label{exampleredundancies}
We consider a $3 \times 3$ symmetric matrix of unknowns, and
the kernel corresponding to the configuration $\{1,2\} \subset \{1,2,3\}$.
Let
\[
\left(
\begin{array}{ccc}
f_{11} & f_{12} \\
f_{21} & f_{22} \\
f_{31} & f_{32} 
\end{array}
\right)
=
\left(
\begin{array}{ccc}
x_{1} & x_{2} & x_{3} \\
x_{2} & x_{4} & x_{5} \\
x_{3} & x_{5} & x_{6}
\end{array}
\right)
\left(
\begin{array}{ccc}
1 & 0 \\
0 & 1 \\
y_{31} & y_{32} 
\end{array}
\right).
\]
We consider the classes of polynomials $f_{12},f_{21}$ modulo ${\langle f_{31}, f_{32}  \rangle}$, deducing the following linear
relation:
$
f_{12}-f_{21} = y_{32}x_{3}-y_{31}x_{5} \equiv y_{31}x_{6}y_{32}-y_{32}x_{6}y_{31} = 0.
$
\end{example}

\subsubsection*{Lagrange systems}

Let $f(A,\iota)$ be the polynomial system defining $\calV_r(A,\iota)$. We set
$c = m(m-r)+\binom{m-r+1}{2}$ {and} $e = \binom{m-r}{2}$, so that \new{$\calV_r$ is defined
by} $c=\sharp f_{red}$
\new{polynomial equations, and $e=\sharp f-c$ is the
number of redundancies eliminated by Lemma \ref{lemma:cleaningrelations}}.
We define, for $M \in \GL_n(\CC)$, the polynomial system $\ell = \ell(A \circ M, \iota)$, given by
the coordinates of the map
\[
  \begin{array}{lrcc}
    \ell : &  \CC^{n} \times \CC^{m(m-r)} \times \CC^{c+e} & \longrightarrow & \CC^{n+m(m-r)+c+e} \\
            &  (x,y,z) & \longmapsto & \left(f(A \circ M, \iota), z^T \jac f(A \circ M, \iota) - (e_1^T, 0) \right),
  \end{array}
\]
where $e_1 \in \QQ^n$ is the first \new{column of the identity matrix $\Id_n$}.
We also define $\calZ(A \circ M, \iota) = \zeroset{\ell(A \circ M, \iota)}$.
\new{When $\calV_r(A \circ M, \iota)$ is smooth and equidimensional, $\calZ(A \circ M, \iota)$
encodes the critical points of the restriction of $\Pi_1(x,y)=x_1$ to $\calV_r(A \circ M, \iota)$.}

\subsubsection*{Output representation}

The output of the algorithm is a rational parametrization $(q_0,\ldots,q_{n+1}) \subset \ZZ[t]$ such
that the finite set defined in \eqref{rationalparametrization} contains at least one point
on the spectrahedron $\spec$.

\subsection{Real root finding for symmetric low rank loci} \label{sec3:ssec2}

We describe the main subroutine {\sf LowRankSym}, which is a variant for
symmetric pencils of the algorithms in \cite{HNS2014,HNS2015a,HNS2015b}.
\new{Its output is a finite set meeting each connected component of $\calD_r \cap \RR^n$}.
It takes advantage of the particular properties of the incidence
varieties over a symmetric low rank locus, as highlighted by
Lemma \ref{lemma:cleaningrelations}.

\subsubsection*{Properties}
We define the following properties for a given $A \in {\mathbb{S}}_m^{n+1}(\QQ)$:

{\it Property $\sfP_1$}. We say that $A$ satisfies $\sfP_1$ if, for all $\iota \subset \{1,\ldots,m\}$,
with $\sharp\iota = m-r$, the incidence variety $\calV_r(A,\iota)$ is either empty or smooth and equidimensional.

{\it Property $\sfP_2$}. We say that $A$ satisfies $\sfP_2$ if, for all $r$, $\calD_r$ has the expected
dimension $n-\binom{m-r+1}{2}$. 
Property $\sfP_2$ holds generically in ${\mathbb{S}}_m^{n+1}(\QQ)$, as shown by Lemma \ref{dimdetvarsym}.

We also define the following property for a polynomial system $f \subset \QQ[x]$ and
a Zariski open set $\zarO \subset \CC^n$:

{\it Property $\sfQ$}. Suppose that $f \subset \QQ[x]$ generates a radical ideal and that it defines an algebraic
set of codimension $c$, and let $\zarO \subset \CC^n$ be a Zariski open set. We say that $f$ satisfies $\sfQ$ in
$\zarO$, if the rank of $\jac f$ is $c$ in $\zeroset{\langle f \rangle} \cap \zarO$.

\subsubsection*{Formal description of {\sf LowRankSym}}

The formal description of our algorithm is given next. \new{It consists of a main algorithm which checks the
genericity hypotheses, and of a recursive sub-algorithm called {\sf LowRankSymRec}}.

\vspace{0.2cm}

{
\begin{center}
\boxed{
  \begin{varwidth}{11cm}
    {\large\begin{center}{${\sf LowRankSym}$}\end{center}}

    \vspace{0.1cm}

    {\bf Input:}
    $A \in {\mathbb{S}}_m^{n+1}(\QQ)$, encoded by the $m(m+1)(n+1)/2$ entries of $A_0, A_1, \ldots, A_n$,
    and an integer $1 \leq r \leq m-1$; \smallskip

    {\bf Output:} Either the empty list $[\,\,]$, if and only if $\calD_r \cap \RR^n = \emptyset$,
    or an error message stating that the genericity assumptions are not satisfied, 
    or a rational parametrization $q=(q_0, \ldots, q_{n+1}) \subset \ZZ[t]$, such
    that for every connected component $\cc \subset \calD_r \cap \RR^n$, with $\cc \cap \calD_{r-1}=\emptyset$,
    there exists $t^* \in \zeroset{q_{n+1}}\cap \RR$ with $(q_1(t^*)/q_0(t^*), \ldots, q_n(t^*)/q_0(t^*))
    \in \cc$. \smallskip \smallskip

    {{\bf Procedure} {\sf LowRankSym}$(A,r)$}
    \begin{enumerate}
    \item \label{sym:step:low:1}
      if $n<\binom{m-r+1}{2}$ then
      \begin{itemize}
      \item \new{if $\dim \calD_r = -1$ then return $[\,\,]$, else return(``the input is not generic'');}
      \end{itemize} 
    \item \label{sym:step:low:2}
      for $\iota \subset \{1, \ldots, m\}$ with $\sharp \iota = m-r$ do
      \begin{itemize}
      \item \label{sym:step:low:2a}
        if {\sf IsReg}($(A, \iota)$) = {\tt false}
        then return(``the input is not generic'');
      \end{itemize}
    \item \label{sym:step:low:3}
      return(${\sf LowRankSymRec}(A,r)$).
    \end{enumerate}

    \vspace{0.1cm}

    {{\bf Procedure} {\sf LowRankSymRec}$(A,r)$}

    \begin{enumerate}
    \item \label{sym:step:lowrec:1}
      choose $M \in \GL_n(\QQ)$;
    \item \label{sym:step:lowrec:2}
      $q \leftarrow [\,\,]$; for $\iota \subset \{1, \ldots, m\}$ with $\sharp\iota=m-r$ do
      \begin{itemize}
      \item \label{sym:step:lowrec:2a}
        $q_\iota \leftarrow  {\sf Image}({\sf Project}({\sf RatPar}(\ell(A \circ M, \iota))),M^{-1})$;
      \item \label{sym:step:lowrec:2b}
        $q \leftarrow {\sf Union}(q,q_\iota)$;
      \end{itemize}
    \item \label{sym:step:lowrec:3}
      choose $\fiber \in \QQ$; $A \leftarrow (A_0+tA_1, A_2, \ldots, A_n)$;
    \item \label{sym:step:lowrec:4}
      $q' \leftarrow {\sf Lift}({\sf LowRankSymRec}(A,r),\fiber)$;
    \item \label{sym:step:lowrec:5}
      return(${\sf Union}(q,q')$).
    \end{enumerate}

  \end{varwidth}
}
\end{center}
}

\vspace{0.2cm}

The routines appearing in the previous algorithm are described next:

\begin{itemize}
\item {\sf IsReg}. {\it Input:} $A \in {\mathbb{S}}_m^{n+1}(\QQ), \iota \subset \{1,\ldots,m\}$;
   {\it Output:} {\tt true} if $\calV_r(A,\iota)$ is empty or smooth and equidimensional of
    codimension $m(m-r) + \binom{m-r+1}{2}$, {\tt false} otherwise.
\item {\sf Project}. {\it Input:} A rational parametrization of $\ell(A \circ M, \iota) \subset \QQ[x,y,z]$;
{\it Output:} an error message if the projection of $\calZ(A\circ M,\iota) \cap \{(x,y,z) \mymid \rank A(M\,x)=r\}$
on the $x-$space is not finite; otherwise a rational parametrization of this projection.
\item \new{{\sf RatPar}. {\it Input:} $\ell(A \circ M, \iota) \subset \QQ[x,y,z]$; {\it Output}: a rational
parametrization of $\ell(A \circ M, \iota)$.}
\item {\sf Image}. {\it Input:} a rational parametrization of a set $\calZ \subset \QQ[x_1, \ldots, x_N]$ and
  a matrix $M \in \GL_N(\QQ)$; {\it Output:} a rational parametrization of $M^{-1}\calZ = \{x \in \CC^N \mymid M\,x \in \calZ\}$.
\item {\sf Union}. {\it Input:} two rational parametrizations encoding $\calZ_1,\calZ_2 \subset \QQ[x_1, \ldots, x_N]$;
{\it Output:} a rational parametrization of $\calZ_1 \cup \calZ_2$.
\item {\sf Lift}.
{\it Input:} a rational parametrization of a set $\calZ \subset \QQ[x_1, \ldots, x_N]$, and $\fiber \in \CC$;
{\it Output:} a rational parametrization of $\{(\fiber,x) \mymid x \in \calZ\}$. 
\end{itemize}

\subsection{Main algorithm: description} \label{sec3:ssec3}
The input of {\sf SolveLMI} is a linear matrix $A \in {\mathbb{S}}_{m}^{n+1}(\QQ)$, and the algorithm makes use
of algorithm {\sf LowRankSym} described previously, as a subroutine.
The formal description is the following.

\vspace{0.2cm}

\begin{center}
\boxed{
{
  \begin{varwidth}{11cm}
    {\large\begin{center}{${\sf SolveLMI}$}\end{center}}

    \vspace{0.1cm}

    {\bf Input:}
    $A \in {\mathbb{S}}_m^{n+1}(\QQ)$, encoded by the $m(m+1)(n+1)/2$ entries of $A_0, A_1, \ldots, A_n$; \smallskip


    {\bf Output:} Either the empty list $[\,\,]$, if and only if $\{x \in \RR^n \mymid A(x) \succeq 0\}$ is empty;
    or an error message stating that the genericity assumptions are not satisfied, or, otherwise, either a vector
    $x^*=(x^*_1, \ldots, x^*_n)$ such that $A(x^*)=0$, or a rational
    parametrization $q=(q_0, \ldots, q_{n+1}) \subset \ZZ[t]$, such that there exists
    $t^* \in \zeroset{q_{n+1}}\cap\RR$ with $A(q_1(t^*)/q_0(t^*), \ldots, q_n(t^*)/q_0(t^*)) \succeq 0$. \smallskip \smallskip


    {\bf Procedure {\sf SolveLMI}$(A)$} 
    \begin{enumerate}
    \item \label{sym:step:main:1}
      $x^* \leftarrow {\sf SolveLinear}(A)$; if $x^* \neq [\,\,]$ then return($x^*$);
    \item \label{sym:step:main:2}
      for $r$ from $1$ to $m-1$ do:
      \begin{itemize}
      \item \label{sym:step:main:2a}
        $q \leftarrow {\sf LowRankSym}(A,r)$;
      \item \label{sym:step:main:2aa}
        if $q=$ ``the input is not generic'' then return($q$);
      \item \label{sym:step:main:2b}
        if $q \neq [\,\,]$ then $b \leftarrow {\sf CheckLMI}(A,q)$;
      \item \label{sym:step:main:2c}
        if $b={\tt true}$ then return($q$);
      \end{itemize}
    \item \label{sym:step:main:3}
      return($[\,\,]$, ``the spectrahedron is empty'').
    \end{enumerate}
  \end{varwidth}
}}
\end{center}

\vspace{0.2cm}

The different subroutines of {\sf SolveLMI} are described next:
\begin{itemize}
\item {\sf SolveLinear}. {\it Input}: $A \in {\mathbb{S}}_m^{n+1}(\QQ)$; {\it Output} the empty list if
  $A(x)=0$ has no solution, otherwise $x^*$ such that $A(x^*)=0$;
\item {\sf CheckLMI}. {\it Input}: $A \in {\mathbb{S}}_m^{n+1}(\QQ)$ and a rational parametrization
  $q \subset \ZZ[t]$; {\it Output}: {\tt true} if $A(q_1(t^*)/q_0(t^*), \ldots, q_n(t^*)/q_0(t^*)) \succeq 0$
  is satisfied for some $t^* \in \zeroset{q_{n+1}}\cap\RR$, {\tt false} otherwise.
\end{itemize}

\subsection{Main algorithm: correctness} \label{sec3:ssec4}
We prove in Theorem \ref{theo:correctness:sym} that {\sf SolveLMI} returns a correct output if genericity
properties on input data and on random parameters chosen during its execution are satisfied; the proof relies
on some preliminary results that are described before. The proofs of these intermediate results \new{share some
techniques with \cite{HNS2014,HNS2015a,HNS2015b} and are given in Section \ref{section4}}.
\subsubsection*{Intermediate results}
The first result is a regularity theorem for the incidence varieties $\calV_r(A,\iota)$. We focus on
property $\sfP_1$ for the input matrix $A$ ({\it cf.} Section \ref{sec3:ssec2}).

\begin{proposition} \label{prop:regularity:sym}
Let $m,n,r \in \NN$, with $0 \leq r \leq m-1$.
\begin{enumerate}
\item
There exists a non-empty Zariski-open set $\zarA \subset
{\mathbb{S}}_m^{n+1}(\CC)$ such that if $A \in \zarA \cap {\mathbb{S}}_m^{n+1}(\QQ)$,
then $A$ satisfies $\sfP_1$;
\item
if $A$ satisfies $\sfP_1$, there exists a non-empty Zariski open set
  $\zarfiber \subset \CC$ such that if $\fiber \in \zarfiber
  \cap \QQ$, \new{then $(A_0+\fiber A_1,A_2,\ldots,A_n)$} satisfies $\sfP_1$.
\end{enumerate}
\end{proposition}
Next, we prove that the projection of $\calZ(A \circ M, \iota) \cap \{(x,y,z) \mymid \rank A(M\,x) = r\}$
over the $x-$space is finite and that this set meets the critical points of the restriction
of the map $\Pi_1 \colon (x,y) \to x_1$ to $\calV_r(A,\iota)$.

\begin{proposition} \label{prop:dimension:sym}
Let $A \in {\mathbb{S}}_m^{n+1}(\QQ)$ satisfy $\sfP_1$. Then there exists a non-empty Zariski open set $\zarM_1 \subset
\GL_n(\CC)$ such that, if $M \in \zarM_1 \cap \MM_{n,n}(\QQ)$, for all $\iota \subset \{1,\ldots,m\}$, with
$\sharp\iota=m-r$, the following holds:
\begin{enumerate}
\item The system $\ell(A \circ M, \iota)$ satisfies $\sfQ$ in $\{(x,y,z) \mymid
  \rank A(M\,x) = r\}$;
\item the projection of $\calZ(A \circ M, \iota) \cap \{(x,y,z) \mymid \rank
  A(M\,x) = r\}$ on the $x-$space is empty or finite;
\item the projection of $\calZ(A \circ M, \iota) \cap \{(x,y,z) \mymid \rank
  A(M\,x) = r\}$ on $(x,y)$ contains the set of critical points of the restriction
  of $\Pi_1 \colon (x,y) \to x_1$ to $\calV_r(A \circ M, \iota) \cap \{(x,y) \mymid
  \rank A(M\,x) = r\}$.
\end{enumerate}
\end{proposition}

Finally, we show, after a generic linear change of variables $x$, closure properties of the projection maps $\pi_i(x)=(x_1,\ldots,x_i)$
restricted to $\calD_r$. Also, in order to compute sample points on the connected components of $\calD_r \cap \RR^n$ not
meeting $\calD_{r-1}$, the next proposition shows that to do that it is sufficient to compute critical points on
the incidence variety $\calV_r$.

\begin{proposition} \label{prop:closure:sym}
Let $A \in {\mathbb{S}}_m^{n+1}(\QQ)$ satisfy $\sfP_1$, and let $d = \dim \calD_r$.
There exists a non-empty Zariski open set $\zarM_2 \subset \GL_n(\CC)$ such
that if $M \in \zarM_2 \cap \MM_{n,n}(\QQ)$, for any connected component
$\cc \subset \calD_r \cap \RR^n$, the following holds:
\begin{enumerate}
\item for $i=1, \ldots, d$, $\pi_i(M^{-1}\cc)$ is closed; further, for $\fiber \in
  \RR$ lying on the boundary of $\pi_1(M^{-1}\cc)$, then $\pi_1^{-1}(\fiber)
  \cap M^{-1} \cc$ is finite;
\item  let $\fiber$ lie on the boundary of $\pi_1(M^{-1}\cc)$: for $x \in \pi_1^{-1}(\fiber)
  \cap M^{-1}\cc$, such that $\rank A(M\,x) = r$, there exists $\iota \subset \{1, \ldots,
  m\}$, with $\sharp\iota=m-r$, and $(x, y) \in \calV_r(A \circ M, \iota)$, such that $\Pi_1(x,y) =
  \fiber$.
\end{enumerate}
\end{proposition}
\subsubsection*{Theorem of correctness}
Let $A \in {\mathbb{S}}_{m,m}^{n+1}(\QQ)$ be the input of {\sf SolveLMI}. We say that hypothesis $\sfH$ holds if:
\begin{itemize}
\item
$A$ and all parameters generated by {\sf SolveLMI} belong to the Zariski open sets defined in Proposition
\ref{prop:regularity:sym}, \ref{prop:dimension:sym} and \ref{prop:closure:sym}, for all recursive steps of {\sf LowRankSym};
\item
$A$ satisfies Property $\sfP_2$.
\end{itemize}

We can now state the correctness theorem for {\sf SolveLMI}.

\begin{theorem}[Correctness of {\sf SolveLMI}] \label{theo:correctness:sym}
Suppose that $\sfH$ holds. Let $\spec = \{x \in \RR^n \mymid A(x) \succeq 0\}$ be the spectrahedron
associated to $A$. Then two alternatives hold:
\begin{enumerate}
\item
  $\spec = \emptyset$:
  hence the output of {\sf SolveLMI} with input $A$ is the empty list;
\item
  $\spec \neq \emptyset$:
  hence the output of {\sf SolveLMI} with input $A$ is either a vector $x^*$ such
  that $A(x^*)=0$, if it exists; or a rational parametrization
  $q = (q_0,\ldots, q_{n+1}) \subset \ZZ[t]$ such that there exists $t^* \in \zeroset{q_{n+1}}\cap\RR$
  with:
  \begin{itemize}
  \item $A(q_1(t^*)/q_0(t^*),\ldots,q_n(t^*)/q_0(t^*)) \succeq 0$ and
  \item $\rank A(q_1(t^*)/q_0(t^*),\ldots,q_n(t^*)/q_0(t^*)) = \minrank$ ({\it cf.} Notation \ref{minrank}).
  \end{itemize}
\end{enumerate}
\end{theorem}

\begin{proof}
Suppose $A(x)=0$ has a solution. Hence, at Step \ref{sym:step:main:1} of {\sf SolveLMI}, {\sf SolveLinear} with input $A$ returns
a vector $x^*$ such that $A(x^*)=0$. We deduce that $x^* \in \spec \neq \emptyset$ and that the rank of $A$ attains its minimum on
$\spec$ at $x^*$. We deduce that, if $A(x)=0$ has at least one solution, the algorithm returns a correct output.

Suppose now that either $\spec$ is empty, or $\minrank \geq 1$. We claim that {\sf LowRankSym} is correct, in
the following sense: with input $(A,r)$, with $A$ satisfying $\sfP_1$, the output is a rational parametrization whose solutions
meet each connected component $\cc$ of $\calD_r \new{\cap \RR^n}$ such that $\cc \cap \calD_{r-1} = \emptyset$.

We assume for the moment this claim, and consider two possible alternatives:
\begin{enumerate}
\item $\spec = \emptyset$. Consequently, {\sf CheckLMI} outputs
  {\tt false} at each iteration of Step \ref{sym:step:main:2b} in
  {\sf SolveLMI}. Thus the output of {\sf SolveLMI}
  is the empty list, and correctness follows.
\item $\spec \neq \emptyset$. 
  Denote by $\cc$ a connected component of $\calD_{\minrank} \cap \RR^n$ such that $\cc \cap \spec \neq \emptyset$.
  By Theorem \ref{theo:spectra}, we deduce that $\cc \subset \spec$, and that $\cc \cap \calD_{\minrank-1} = \emptyset$.
  Let $q$ be the output of {\sf LowRankSym} at Step \ref{sym:step:main:2a} of {\sf SolveLMI}. \new{By our claim},
  $q$ defines a finite set whose solutions meet $\cc$, hence $\spec$. Consequently, {\sf CheckLMI} returns {\tt true}
  at Step \ref{sym:step:main:2b}, and hence the algorithm stops returning the correct output $q$.
\end{enumerate}
We end the proof by showing that {\sf LowRankSym} is correct. This is straightforwardly implied by the correctness of the recursive
subroutine {\sf LowRankSymRec}, which is proved below by using induction on the number of variables $n$.

For $n < \binom{m-r+1}{2}$, since $\sfH$ holds, then $A$ satisfies $\sfP_r$. Hence $\calD_r$ is empty, and {\sf LowRankSym} returns the
correct answer $[\,\,]$ (the empty list).

Let $n \geq \binom{m-r+1}{2}$, and let $(A,r)$ be the input. The induction hypothesis implies that for any $\tilde{A} \in {\mathbb{S}}_m^{n}(\QQ)$
satisfying $\sfP_1$, then {\sf LowRankSymRec} with input $(\tilde{A},r)$ returns a rational parametrization of a finite set meeting each
connected component $\tilde{\cc} \subset \tilde{\calD}_r \new{\cap \RR^{n-1}}$ such that $\tilde{\cc} \cap \tilde{\calD}_{r-1} =\emptyset$,
with $\tilde{\calD}_r = \{x \in \RR^{n-1} : \rank \tilde{A}(x) \leq r\}$.
Let $\cc \subset \calD_r \new{\cap \RR^n}$ be a connected component with $\cc \cap
\calD_{r-1} = \emptyset$, and let $M$ be the matrix chosen at Step
\ref{sym:step:lowrec:1}. Hence, since $\sfH$ holds, by Proposition
\ref{prop:closure:sym} the set $\pi_1(M^{-1}\cc)$ is closed in $\RR$. There are
two possible scenarios.

{\it First case.} Suppose that $\pi_1(M^{-1}\cc) = \RR$,
let $\fiber \in \QQ$ be the rational number chosen
at Step \ref{sym:step:lowrec:3} of {\sf LowRankSymRec}, and let $\tilde{A} = (A_0+\fiber A_1,
A_2, \ldots, A_n) \in {\mathbb{S}}_m^n(\QQ)$. We deduce that $\pi_1^{-1}(\fiber) \cap M^{-1}\cc
\neq \emptyset$ is the union of some connected components
of the algebraic set $\tilde{\calD}_{r}=\{x \in \RR^{n-1} \mymid \rank \tilde{A}(x) \leq
r\}$ not meeting $\tilde{\calD}_{r-1}$. Also, since $A$ satisfies $\sfP_1$,
so does $A \circ M$; by Proposition \ref{prop:regularity:sym},
then $\tilde{A}$ satisfies $\sfP_1$.
By the induction assumption, {\sf LowRankSymRec} with input $(\tilde{A},r)$
returns at least one point in each connected component $\tilde{\cc} \subset \tilde{\calD}_r \new{\cap \RR^{n-1}}$
not meeting $\tilde{\calD}_{r-1}$, hence one point in $\cc$ by applying
the subroutine {\sf Lift} at Step \ref{sym:step:lowrec:4}. Correctness follows.

{\it Second case.} Otherwise, $\pi_1(M^{-1}\cc) \neq \RR$ and, since it is a closed set,
its boundary is non-empty. Let $t$ belong to the boundary of $\pi_1(M^{-1}\cc)$,
and suppose w.l.o.g. that $\pi_1(M^{-1}\cc) \subset [t, +\infty)$.
Hence $t$ is the minimum of the restriction of the map $\pi_1$ to $M^{-1}\cc$.
By Proposition \ref{prop:closure:sym}, the set $\pi_1^{-1}(\fiber) \cap M^{-1}\cc$ is non-empty and finite,
and for all $x \in \pi_1^{-1}(\fiber) \cap M^{-1}\cc$, $\rank A(M\,x) = r$
(indeed, for $x \in M^{-1}\cc$, then $M\,x \in \cc$ and hence $M\,x \notin \calD_{r-1} \cap \RR^n$).
Fix $x \in \pi_1^{-1}(\fiber) \cap M^{-1}\cc$. By Proposition \ref{prop:closure:sym},
there exists $\iota$ and $y \in \CC^{m(m-r)}$ such that $(x,y) \in \calV_r(A
\circ M, \iota)$. Also, by Proposition \ref{prop:regularity:sym}, the set
$\calV_r(A\circ M, \iota)$ is smooth and equidimensional.
One deduces that $(x,y)$ is a critical point of the restriction of $\Pi_1 \colon (x,y) \to x_1$
to $\calV_r(A \circ M, \iota)$ and that there exists $z$ such that
$(x,y,z) \in \calZ(A \circ M, \iota)$. Hence, at Step
\ref{sym:step:lowrec:2a}, the routine {\sf LowRankSymRec} outputs a rational
parametrization $q_\iota$, among whose solutions the vector $x$ lies.
\end{proof}

\section{Proof of intermediate results}
\label{section4}

\subsection{Proof of Proposition \ref{prop:regularity:sym}} We prove Assertion 1 and 2 separately.

\begin{proof}[of Assertion 1]
Suppose w.l.o.g. that $M=\Id_n$. For $\iota \subset \{1, \ldots, m\}$, with $\sharp\iota=m-r$, let $f_{red} \subset \QQ[x,y]$ be
the system defined in Lemma \ref{lemma:cleaningrelations}. We prove that there exists a non-empty Zariski open set $\zarA_\iota
\subset {\mathbb{S}}_m^{n+1}(\CC)$ such that, if $A \in \zarA_\iota \cap {\mathbb{S}}_m^{n+1}(\QQ)$, $f_{red}$ generates a radical ideal and
$\zeroset{f_{red}}$ is empty or equidimensional, of codimension $\sharp f_{red} = m(m-r)+\binom{m-r+1}{2}$. We deduce that, for
$A \in \zarA_\iota$, $A$ satisfies $\sfP_1$, and we conclude by defining $\zarA = \cap_\iota \zarA_\iota$ (non-empty and Zariski open).

Suppose w.l.o.g. that $\iota = \{1, \ldots, m-r\}$. We consider the map
\[
  \begin{array}{lrcc}
  \varphi : &  \CC^{n+m(m-r)} \times {\mathbb{S}}_m^{n+1}(\CC) & \longrightarrow & \CC^{m(m-r) + \binom{m-r+1}{2}} \\
            &  (x,y,A) & \longmapsto& f_{red}
  \end{array}
\]
and, for a fixed $A \in {\mathbb{S}}_m^{n+1}(\CC)$, its section map
$\varphi_A \colon \CC^{n+m(m-r)} \to \CC^{m(m-r) + \binom{m-r+1}{2}}$
defined by $\varphi_A(x,y) = \varphi(x,y,A)$. Remark that, for any
$A$, $\zeroset{\varphi_A}=\calV_r(A, \iota)$.

Suppose $\varphi^{-1}(0) = \emptyset$:
this implies that, for all $A \in {\mathbb{S}}_m^{n+1}(\CC)$,
$\zeroset{f_{red}} = \calV_r(A, \iota) = \emptyset$,
that is $A$ satisfies $\sfP_1$ for $A \in \zarA_\iota = {\mathbb{S}}_m^{n+1}(\CC)$.

If $\varphi^{-1}(0) \neq \emptyset$, we prove below that $0$
is a regular value of $\varphi$. We conclude that by Thom's Weak
Transversality Theorem \cite[Section 4.2]{SaSc13} there exists
a non-empty and Zariski open set $A_\iota\subset{\mathbb{S}}_m^{n+1}(\CC)$ such that 
if $A \in \zarA_\iota \cap {\mathbb{S}}_m^{n+1}(\QQ)$, 0 is a regular value
of $\varphi_A$. Hence, by applying the Jacobian criterion ({\it cf.}
\cite[Theorem 16.19]{Eisenbud95}) to the polynomial system $f_{red}$,
we deduce that for $A \in \zarA_\iota \cap {\mathbb{S}}_m^{n+1}(\QQ)$, $\calV_r(A,\iota)$
is smooth and equidimensional of codimension $\sharp f_{red}$.

Let $\jac \varphi$ be the Jacobian matrix of $\varphi$: it contains
the derivatives of polynomials in $f_{red}$ with respect to variables $x,y,A$.
We denote by $a_{\ell, i, j}$ the variable encoding the $(i,j)-$th entry of the matrix $A_\ell, \ell=0, \ldots, n$.
We isolate the columns of $\jac \varphi$ corresponding to:
\begin{itemize}
\item the derivatives with respect to variables $\{a_{0,i,j} \mymid i \leq m-r \,\,\, \text{or} \,\,\, j \leq m-r\}$;
\item the derivatives with respect to variables $y_{i,j}$ such that $i \in \iota$.
\end{itemize}
Let $(x,y,A) \in \varphi^{-1}(0)$, and consider the evaluation of $\jac \varphi$ at $(x,y,A)$. The above columns contain the
following non-singular blocks:
\begin{itemize}
\item
the derivatives w.r.t. $\{a_{0,i,j} \mymid i \leq m-r \,\,\, \text{or} \,\,\, j \leq m-r\}$
of the entries of $A(x) Y(y)$ after the substitution $Y_{\iota} \leftarrow \Id_{m-r}$, that is $\Id_{(m-r)(m+r+1)/2}$;
\item
the derivatives w.r.t. $\{y_{i,j} \mymid i \in \iota\}$ of polynomials in $Y_\iota-\Id_{m-r}$, that is $\Id_{(m-r)^2}$.
\end{itemize}
Hence, the above columns define a maximal non-singular
sub-matrix of $\jac \varphi$ at $(x,y,A)$, of size
$m(m-r)+\binom{m-r+1}{2}=\sharp f_{red}$ (indeed, remark that the entries of $Y_\iota-\Id_{m-r}$ do not depend on variables
$a_{0,i,j}$).
Since $(x,y,A) \in \varphi^{-1}(0)$ is arbitrary, we deduce
that $0$ is a regular value of $\varphi$, and we conclude.
\end{proof}

\begin{proof}[of Assertion 2]
Fix $\iota \subset \{1, \ldots, m\}$ with $\sharp\iota = m-r$.
Since $A$ satisfies $\sfP_1$, $\calV_r(A, \iota)$
is either empty or smooth and equidimensional of codimension
$m(m-r)+\binom{m-r+1}{2}$.
Suppose first that $\calV_r = \emptyset$. Hence for all
$\fiber \in \CC$, $\calV_r \cap \{x_1-t=0\} = \emptyset$,
and we conclude by defining $\zarfiber = \CC$.
Otherwise, consider the restriction of the projection map
$\Pi_1 : (x,y) \to x_1$ to $\calV_r(A, \iota)$. By Sard's
Lemma \cite[Section 4.2]{SaSc13}, the set of critical
values of the restriction of $\pi_1$ to $\calV_r(A, \iota)$
is included in a finite subset $\mathcal{H} \subset \CC$.
We deduce that, for $\fiber \in \zarfiber = \CC \setminus \mathcal{H}$,
then $(A_0+\fiber A_1,A_2,\ldots,A_n)$ satisfies $\sfP_1$.
\end{proof}
\begin{corollary} \label{cor:completeInt}
Let $\zarA \subset {\mathbb{S}}_m^{n+1}(\QQ)$ be as in Proposition \ref{prop:regularity:sym},
and let $A \in \zarA$. Then for every $\iota \subset \{1,\ldots,m\}$ with $\sharp\iota=m-r$,
the ideal $\langle f_{red} \rangle = \langle f \rangle$ is radical, and $\calV_r(A,\iota)$
is a complete intersection of codimension $\sharp f_{red}$.
\end{corollary}

\begin{proof}
We recall from the proof of Assertion 1 of Theorem \ref{prop:regularity:sym} that,
for $A \in \zarA$, the rank of the Jacobian matrix of $f_{red}$ is $\sharp f_{red} =
m(m-r)+\binom{m-r+1}{2}$ at every point of $\calV_r(A,\iota)$. By the Jacobian
criterion \cite[Theorem 16.19]{Eisenbud95}, the ideal $\langle f_{red} \rangle$
is radical and the algebraic set $\zeroset{f_{red}} = \calV_r(A,\iota)$ is smooth
and equidimensional of codimension $\sharp f_{red}$. Hence $\ideal{\calV_r(A,\iota)}=\langle f_{red} \rangle$
can be generated by a number of polynomials equal to the codimension of
$\calV_r(A,\iota)$, and we conclude.
\end{proof}

\subsection{Proof of Proposition \ref{prop:dimension:sym}}


\subsubsection*{Local equations of $\calV_r(A,\iota)$}
Suppose \new{$A$ is a (not necessarily symmetric) linear matrix}. \new{Let us give a local description of the algebraic sets
$\calD_r$ and $\calV_r$ ({\it cf.} also \cite[Section 5]{HNS2014})}.
Consider first the locally closed set $\calD_r \setminus \calD_{r-1} = \{x \in \CC^n \mymid \rank A(x) = r\}.$ This is given
by the union of sets $\calD_r \cap \{x \in \CC^n \mymid \det N(x) \neq 0\}$ where $N$ runs over all $r \times r$
sub-matrices of $A(x)$. Fix $\iota \subset \{1, \ldots, m\}$ with $\sharp\iota=m-r$. Let $N$ be the upper left $r \times r$
sub-matrix of $A(x)$, and consider the corresponding block division of $A$:
\begin{equation} \label{blockdivision}
A =
\left(
\begin{array}{cc}
N & Q \\
P^T & R
\end{array}
\right)
\end{equation}
with $P,Q \in \MM_{r,m-r}(\QQ)$ and $R \in \MM_{m-r,m-r}(\QQ)$.
Let $\QQ[x,y]_{\det N}$ be the local ring obtained by localizing $\QQ[x,y]$ at
$\left\langle \det N \right\rangle$. Let $Y^{(1)}$ (resp. $Y^{(2)}$) be the matrix
obtained by isolating the first $r$ (resp. the last $m-r$) rows of $Y(y)$.
Hence, the local equations of $\calV_r$ in $\{(x,y) \mymid \det N(x) \neq 0\}$
are given by:
\begin{equation} \label{local-equations}
Y^{(1)}+N^{-1}QY^{(2)} = 0, \qquad
\Sigma(N) Y^{(2)} = 0, \qquad
Y_\iota - \Id_{m-r} = 0,
\end{equation}
where $\Sigma(N) = R-P^T N^{-1}Q$ is the Schur complement of $N$ in $A$. This follows
from the following straightforward equivalence holding in the local ring $\QQ[x,y]_{\det N}$
\new{({\it cf.} also \cite[Lemma 13]{HNS2014})}:
\begin{equation*}
A(x)Y(y)=0
\qquad 
\text{iff} \qquad
\left(\begin{array}{cc} \Id_r & 0 \\ -P^T & \Id_{m-r} \end{array} \right)
\left( \begin{array}{cc} N^{-1} & 0 \\ 0 & \Id_{m-r} \end{array} \right)
\left( \begin{array}{cc} N & Q \\ P^T & R \end{array} \right) Y(y) = 0.
\end{equation*}

\subsubsection*{Intermediate lemma}
Let $w \in \CC^n$ be a non-zero vector, and consider the projection map induced by $w$:
$
\Pi_w \colon (x_1, \ldots, x_n, y) \mapsto w_1x_1+\cdots+w_nx_n.
$

For $A \in \zarA$ (given by Proposition \ref{prop:regularity:sym}),
for all $\iota$ as above, the critical points of the restriction
of $\Pi_w$ to $\calV_r(A, \iota)$ are encoded by the
polynomial system $(f,g,h)$ where
\begin{equation}
\label{lag:syst:global}
f=f(A,\iota), \qquad \qquad
(g,h) =
z^T
\left(
\begin{array}{c}
D f \\
D \Pi_w
\end{array}
\right)
=
z^T
\left(
\begin{array}{cc}
D_x f & D_y f \\
w^T & 0
\end{array}
\right),
\end{equation}
and $z = (z_1, \ldots, z_{c+e}, 1)$ is a vector of
Lagrange multipliers. Indeed, equations induced by
$(g,h)$ imply that the vector $w$ is normal to the
tangent space of $\calV_r$ at $(x,y)$.

We prove an intermediate lemma towards Proposition \ref{prop:dimension:sym}.
\begin{lemma}
\label{lemma:intermediate:sym}
Let $A \in {\mathbb{S}}_m^{n+1}(\QQ)$ satisfy $\sfP_1$.
Then there exists a non-empty Zariski open set $\mathscr{W} \subset
\CC^n$ such that, if $w \in \mathscr{W} \cap \QQ^n$, for all $\iota \subset \{1,\ldots,m\}$,
with $\sharp\iota=m-r$, the following holds:
\begin{enumerate}
\item the system $(f,g,h)$ in \eqref{lag:syst:global} satisfies $\sfQ$
  in $\{(x,y,z) \mymid \rank A(x) = r\}$;
\item the projection of $\zeroset{f,g,h} \cap \{(x,y,z) \mymid \rank
  A(x) = r\}$ on the $x-$space is empty or finite;
\item the projection of $\zeroset{f,g,h} \cap \{(x,y,z) \mymid \rank
  A(x) = r\}$ on $(x,y)$ contains the set of critical points of
  the restriction of $\Pi_w$ to $\calV_r \cap \{(x,y) \mymid \rank A(x) = r\}$.
\end{enumerate}
\end{lemma}

\begin{proof}[of Assertion 1]
The strategy relies on applying Thom Weak Transversality Theorem and
the Jacobian criterion, as in the proof of Proposition \ref{prop:regularity:sym}.
\new{The similar passages will be only sketched.}

We claim (and prove below) that, given a $r \times r$ sub-matrix $N$ of $A$,
there exists a non-empty Zariski open set $\mathscr{W}_N \subset \CC^n$ such that,
for $w \in \mathscr{W}_N$, $(f,g,h)$ satisfies $\sfQ$ in $\{(x,y,z) \mymid \det N \neq 0\}$.
We deduce Assertion 1 by defining $\mathscr{W}=\bigcap_N \mathscr{W}_N$, where $N$
runs over all $r \times r$ sub-matrices of $A(x)$.

Let $U_\iota \in \CC^{(m-r) \times m}$ be the boolean matrix such that
$U_\iota Y(y) = Y_\iota$, and let $U_\iota = (U^{(1)}_\iota \mid U^{(2)}_\iota)$
be the subdivision with $U^{(1)}_\iota \in \CC^{(m-r) \times r}$ and
$U^{(2)}_\iota \in \CC^{(m-r) \times (m-r)}$.
The third equation in \eqref{local-equations} equals $U_\iota Y(y) - \Id_{m-r} = 0.$
From \eqref{local-equations} we deduce the equality
\[
\Id_{m-r} = U^{(1)}_\iota Y^{(1)} + U^{(2)}_\iota Y^{(2)} = (U^{(2)}_\iota-U^{(1)}_\iota N^{-1}{Q})Y^{(2)}
\]
and hence that both $Y^{(2)}$ and $U^{(2)}_\iota-U^{(1)}_\iota N^{-1}P$
are non-singular in $\QQ[x,y]_{\det N}$. We deduce that the above local equations of
$\calV_r$ are equivalent to
\[
Y^{(1)}+N^{-1}QY^{(2)} = 0, \qquad
\Sigma(N) = 0, \qquad
Y^{(2)}-(U^{(2)}_\iota-U^{(1)}_\iota N^{-1}P)^{-1} = 0
\]
in $\QQ[x,y]_{\det N}$. We collect the above equations in a system $\tilde{f}$, of length $c+e$.
Hence, the Jacobian matrix of $\tilde{f}$ is
\[
\jac \tilde{f} =
\left(
\begin{array}{cc}
D_x[\Sigma(N)]_{i,j} & 0_{(m-r)^2 \times m(m-r)} \\
\star & 
\begin{array}{cc}
\Id_{r(m-r)} & \star \\
0 & \Id_{(m-r)^2}
\end{array}
\end{array}
\right).
\]
Since $A$ satisfies $\sfP_1$, the rank of $\jac \tilde{f}$ is constant
and equal to $c$ if evaluated at $(x,y) \in \zeroset{\tilde{f}}
= \calV_r(A, \iota) \cap \{(x,y) \mymid \det N \neq 0\}$.
{Similarly to \eqref{lag:syst:global},} we define
\[
(\tilde{g},\tilde{h}) = z^T
\left(
\begin{array}{cc}
\jac \tilde{f} \\
w^T \,\,\,\,\,  0
\end{array}
\right)
\]
with $z = (z_1, \ldots, z_{c+e}, 1)$.
{The polynomial relations $\tilde{h}_i=0$ implies that $z_{(m-r)^2+i}=0$, for $i=1, \ldots, m(m-r)$,}
and hence they can be eliminated, together with the variables $z_{(m-r)^2+i}, i=1, \ldots, m(m-r)$.
Hence, one can consider the equivalent equations $(\tilde{f}, \tilde{g}, \tilde{h})$ where the last
$m(m-r)$ variables $z$ do not appear in $\tilde{g}$.

Let us define the map $\varphi \colon \CC^{n+c+e+m(m-r)} \times \CC^{n} \to \CC^{n+c+e+m(m-r)}$,
$\varphi(x,y,z,w) = (\tilde{f}, \tilde{g}, \tilde{h})$, and 
for $w \in \CC^n$, its section map $\varphi_w \colon (x,y,z) \mapsto {\varphi}(x,y,z,w)$.

If $\varphi^{-1}(0) = \emptyset$, we define $\mathscr{W}_N = \CC^n$ and conclude.
Otherwise, let $(x,y,z,w) \in \varphi^{-1}(0)$. Remark that polynomials in $\tilde{f}$
just depend on $(x,y)$, hence their contribution in $\jac \varphi (x,y,z,w)$ is
the block $\jac \tilde{f}$, whose rank is $c$, since $(x,y) \in \calV_r$. Hence,
we deduce that the row-rank of $\jac \varphi$ at $(x,y,z,w)$ is at most $n+c+m(m-r)$.
Further, by isolating the columns corresponding to the derivatives with respect to $x,y$,
to $w_1, \ldots, w_n$, and to $z_{(m-r)^2+i}, i=1, \ldots, m(m-r)$, one obtains a $(n+c+e+m(m-r))
\times (2n+2m(m-r))$ sub-matrix of $\jac \varphi$ of rank $n+c+m(m-r)$.
\end{proof}

\begin{proof}[of Assertion 2]
From Assertion 1 we deduce that the locally closed set $\mathcal{E}=\zeroset{f,g,h} \cap \{(x,y,z)
\mymid \rank A(x) = r\}$ is empty or $e-$equidimensional. If it is empty, we are done. Suppose that
it is $e-$equidimensional. Consider the projection map
$\pi_x(x,y,z)=x$, and its restriction to $\mathcal{E}$. Let $x^* \in \pi_x(\mathcal{E})$.
Then $\rank A(x^*) = r$ and there exists a unique $y \in \CC^{m(m-r)}$
such that $f(x^*,y) = 0$. Hence the fiber $\pi_x^{-1}(x^*)$ is isomorphic
to the linear space
$
\{(z_1, \ldots, z_{c+e}) \mymid (z_1, \ldots, z_{c+e}) \jac f = (w^T, 0)\}.
$
Since the rank of $\jac f$ is $c$, one deduces that
$\pi_x^{-1}(x^*)$ is a linear space of dimension $e$, and by the Theorem
on the Dimension of Fibers \cite[Sect. 6.3, Theorem 7]{Shafarevich77}
we deduce that $\pi_x(\mathcal{E})$ has dimension $0$.
\end{proof}

\begin{proof}[of Assertion 3]
Since $\calV_r \cap \{(x,y) \mymid \rank A(x) = r\}$ is smooth and equidimensional, by
\cite[Lemma 3.2.1]{SaSc13}, for $w \neq 0$, the set $\crit(\Pi_w, \calV_r)$ coincides with
the set of points $(x,y) \in \calV_r$ such that the matrix
\[
\jac (f, \Pi_w) =
\left(
\begin{array}{c}
\jac f \\
\jac \Pi_w
\end{array}
\right)
\]
has a rank $\leq c$. In particular there exists $z=(z_1, \ldots, z_{c+e}, z_{c+e+1}) \neq 0$,
such that $z^T \jac (f, \Pi_w) = 0$. One can exclude that
$z_{c+e+1} = 0$, since this implies that $\jac f$ has a non-zero vector
in the left kernel, which contradicts the fact that $A$ satisfies $\sfP_1$.
Hence without loss of generality we deduce that $z_{c+e+1} = 1$, and we conclude.
\end{proof}


\subsubsection*{Proof of the proposition}
We can finally deduce the proof of Proposition \ref{prop:dimension:sym}.

\begin{proof}[of Proposition \ref{prop:dimension:sym}]
Define $\zarM_1$ as the set of matrices $M \in \GL_n(\CC)$ such that the first row of $M^{-1}$ is contained in $\mathscr{W}$
(defined in Lemma \ref{lemma:intermediate:sym}). The proof of all assertions follows from Lemma \ref{lemma:intermediate:sym}
since, for $M \in \zarM_1$, one gets
\begin{equation}
\label{equality}
\left(
\begin{array}{c}
  \jac f(A \circ M, \iota) \\
  e^T_1 \quad 0 \; \cdots \; 0 \\
\end{array}
\right)
=
\left(
\begin{array}{cc}
  \jac f(A, \iota) \circ M \\
  w^T \quad 0 \; \cdots \; 0 \\
\end{array}
\right)
\left(
\begin{array}{cc}
M & 0 \\
0 & \Id_{m(m-r)}
\end{array}
\right),
\end{equation}
where $w^T$ is the first row of $M^{-1}$. Indeed, for
$z=(z_1, \ldots, z_{c+e})$, we deduce from the
previous relation that the set of solutions to the equations
\begin{equation} \label{firstsystem}
f(A, \iota) = 0, \qquad
z^T \jac f(A, \iota) = (w^T, 0)
\end{equation}
is the image of the set of solutions of
\begin{equation} \label{secondsystem}
f(A \circ M, \iota) = 0, \qquad
z^T \jac f(A \circ M, \iota) = (e_1^T, 0)
\end{equation}
by the linear map
\[
\left(
\begin{array}{c}
x \\
y \\
z
\end{array}
\right)
\mapsto
\left(
\begin{array}{ccc}
M^{-1} & 0 & 0 \\
0 & \Id_{m(m-r)} & 0 \\
0 & 0 & \Id_{c+e}
\end{array}
\right)
\left(
\begin{array}{c}
x \\
y \\
z
\end{array}
\right).
\]
This last fact is straightforward since from \eqref{equality}
we deduce that system \eqref{secondsystem} is equivalent to
$
f(A \circ M, \iota) = 0,
z^T \left( \jac f(A, \iota) \circ M \right) = (w^T, 0).
$
Hence the three assertions of Proposition \ref{prop:dimension:sym}
are straightforwardly deduced by those of Lemma \ref{lemma:intermediate:sym}.
\end{proof}

\subsection{Proof of Proposition \ref{prop:closure:sym}}
For the proof of Assertion 1 of Proposition \ref{prop:closure:sym}, we need to recall some notation
from \cite[Sec.\,5]{HNS2014}. Let $\calZ \subset \CC^n$ be an algebraic set of dimension $d$. Its
equidimensional component of dimension $p$, for $0 \leq p \leq d$, is $\Omega_p(\calZ)$. We define
$
\zarS(\calZ) = \Omega_0(\calZ) \cup \cdots \cup \Omega_{d-1}(\calZ) \cup \sing{\Omega_d{\calZ}},
$
where we recall that $\sing{\calV}$ denotes the singular locus of an algebraic set $\calV$, and
\[
\zarC(\pi_i,\calZ) = \Omega_0(\calZ) \cup \cdots \cup \Omega_{i-1}(\calZ) \cup \bigcup_{r = i}^d \crit(\pi_i,\reg{\Omega_r{\calZ}}),
\]
Here $\reg{\calV}$ denotes $\calV \setminus \sing{\calV}$, $\pi_i \colon x \mapsto (x_1, \ldots, x_i)$
and $\crit(g,\calV)$ the set of critical points of the restriction of a map $g$ to $\calV$. For $M \in
\GL_n(\CC)$ we recursively define
\begin{itemize}
\item ${\mathcal O}_d(M^{-1}\setZ)=M^{-1}\setZ$;
\item ${\mathcal O}_i(M^{-1}\setZ)=\scS({\mathcal O}_{i+1}(M^{-1}\setZ))
\cup \scC(\pi_{i+1},  {\mathcal O}_{i+1}(M^{-1}\setZ)) \cup
\scC(\pi_{i+1},M^{-1}\setZ)$ for $i=0, \ldots, d-1$.
\end{itemize}
In \cite[Prop.\,17]{HNS2014} we proved that
when $M$ is generic in $\GL_n(\CC)$ (that is, it lies out of a proper algebraic set) the algebraic sets
$\calO_i(M^{-1}\calZ)$ have dimension at most $i$ and are in Noether position with respect to $x_1,\ldots,x_i$
({\it cf.} \cite{Shafarevich77,Eisenbud95} for a background in Noether position). We used this
fact in \cite[Prop.\,18]{HNS2014} to prove closure properties of the restriction of maps $\pi_i, i=1, \ldots, d$,
to the connected components of $\calZ \cap \RR^n$.

\begin{proof}[of Assertion 1]
We denote by $\zarM_2 \subset \GL_n(\CC)$ the non-empty Zariski open set defined in
\cite[Prop.\,17]{HNS2014}, for the algebraic set $\calD_r$. Hence, for $M \in \zarM_2$,
we deduce by \cite[Prop.\,18]{HNS2014} that for $i=1, \ldots, d$, and for any connected
component $\cc \subset \calD_r \cap \RR^n$, the boundary of $\pi_i(M^{-1} \cc)$ is contained in
$\pi_i(\mathcal{O}_{i-1}(M^{-1} \calD_r) \cap M^{-1} \cc) \subset \pi_i(M^{-1}\cc)$, and hence
that $\pi_i(M^{-1}\cc)$ is closed.
Moreover, let $\cc \subset \calD_r \cap \RR^n$ be a connected component and let
$t \in \RR$ be in the boundary of $\pi_1(M^{-1}\cc)$. Then \cite[Lemma\,19]{HNS2014} implies
that $\pi_1^{-1}(t) \cap M^{-1}\cc$ is finite.
\end{proof}

\begin{proof}[of Assertion 2]
Let $M \in \zarM_2$. Consider the open set 
\[
\zarO = \{(x,y) \in \CC^{n+m(m-r)} \mymid \rank A(M\,x) = r, \, \rank Y(y) = m-r\}.
\]
Let $\Pi_x \colon \CC^{n+m(m-r)} \to \CC^n$, $\Pi_x(x,y)=x$. Then $\Pi_x(\zarO)$ is the locally closed set 
$M^{-1}(\calD_r \setminus \calD_{r-1})=\{x \in \CC^n \mymid \rank A(M\,x) = r\}.$ We consider the restriction
of polynomial equations $A(M\,x) Y(y) = 0$ to $\zarO$. By definition of $\zarO$,
we can split the locally closed set $\zarO \cap \zeroset{A(M\,x) Y(y)}$ into the
union
\[
\zarO \cap \zeroset{A(M\,x) Y(y)} = \bigcup_{{\begin{array}{c} \iota \subset \{1, \ldots, m\} \\ \sharp\iota=m-r \end{array}}}
\Big( \zarO_{\iota} \cap \zeroset{A(M\,x) Y(y)} \Big),
\]
where $\zarO_\iota = \{(x,y) \mymid \det Y_\iota \neq 0\}$.

Let $\cc$ be a connected component of $\calD_r \cap \RR^n$.
Let $\fiber$ lie in the frontier of $\pi_1(M^{-1}\cc)$, and $x \in \pi_1^{-1}(t) \cap M^{-1}\cc$
with $\rank A(M\,x) = r$. Hence there exists $\iota \subset \{1, \ldots, m\}$
such that $x$ lies in the projection of $\calV_r(A \circ M, \iota)$ on the $x-$space.
Hence there exists $y$ such that $(x,y) \in \calV_r(A \circ M,\iota)$ and  such that $\pi_1(x,y) = t$. 
%
\end{proof}

\section{Complexity analysis}
\label{sec4}
Our next goal is to estimate the \new{arithmetic} complexity of algorithm {\sf SolveLMI}, \new{that
is} the number of arithmetic operations performed over $\QQ$. This will essentially rely on the
complexities of state-of-the-art algorithms computing rational parametrizations.

\subsection{Output degree estimates}
\label{sec4:ssec1}
We start by computing Multilinear B\'ezout bounds ({\it cf.} \cite[Ch.\,11]{SaSc13}) on the output
degree.
\begin{proposition} \label{sym:prop:degree}
Let $A \in {\mathbb{S}}_m^{n+1}$ be the input of {\sf SolveLMI}, and $0 \leq r \leq m-1$. Let $p_r=(m-r)(m+r+1)/2$.
If $\sfH$ holds, for all $\iota \subset \{1, \ldots,m\}$, with $\sharp\iota=m-r$, the degree of the
parametrization $q_\iota$ computed at Step \ref{sym:step:lowrec:2a} of {\sf LowRankSymRec}
is at most
\[
\theta(m,n,r) = \sum_{k \in \calG_{m,n,r}}\binom{p_r}{n-k}\binom{n-1}{k+p_r-1-r(m-r)}\binom{r(m-r)}{k},
\]
with $\calG_{m,n,r} = \{k \mymid \max\{0,n-p_r\} \leq k \leq \min\{n-\binom{m-r+1}{2}, r(m-r)\}\}$.
Moreover, for all $m,n,r$, $\theta(m,n,r)$ is bounded above by $\binom{p_r+n}{n}^3$.
\end{proposition}

\begin{proof}
We can simplify the system $f=f(A,\iota)$ to a system of $p_r$ bilinear equations with respect to variables
$x=(x_1, \ldots, x_n)$ and $y=(y_{m-r+1,1},\ldots,y_{m,m-r})$. Indeed, by Lemma \ref{lemma:cleaningrelations},
$\calV_r(A,\iota)$ is defined by $Y_\iota-\Id_{m-r}=0$ and by $m(m-r)-e=p_r$ entries of $A(x)Y(y)$, where
$e=\binom{m-r}{2}$. Hence we can eliminate equations $Y_\iota-\Id_{m-r}=0$ and the corresponding variables
$y_{i,j}$. Consequently, the Lagrange system can be also simplified, by admitting only $p_r$ Lagrange
multipliers $z$.
We can also eliminate the first Lagrange multiplier $z_1$ (since $z \neq 0$, one can assume $z_1=1$) and
impose a rank defect on the truncated Jacobian matrix obtained by $\jac f$ by eliminating the first column
(that containing the derivatives with respect to $x_1$).

The bound $\theta(m,n,r)$, by \cite[Ch.\,11]{SaSc13}, equals the coefficient of $s_x^ns_y^{r(m-r)}s_z^{p_r-1}$
in $(s_x+s_y)^{p_r}(s_y+s_z)^{n-1}(s_x+s_z)^{r(m-r)}$. This can be easily obtained by writing down such an expansion
and solving the associated linear system forcing the constraints on the exponents of the monomials.
The result is exactly the claimed closed formula.
The estimate $\theta(m,n,r) \leq \binom{p_r+n}{n}^3$ is obtained by applying the standard formula:
\[
\binom{a+b}{a}^3 = \sum_{i_1,i_2,i_3=0}^{\min(a,b)}\binom{a}{i_1}\binom{b}{i_1}\binom{a}{i_2}\binom{b}{i_2}\binom{a}{i_3}\binom{b}{i_3}
\]
with $a=n$ and $b=p_r$.
\end{proof}

\begin{corollary} \label{corollary:degree:sym}
\new{With the hypotheses and notations of Proposition \ref{sym:prop:degree}},
the sum of the degrees of the rational parametrizations computed during {\sf SolveLMI} is bounded above by
$\sum_{r \leq \minrank} \binom{m}{r} \theta(m,n,r).$ The degree of the rational parametrization whose solutions intersect $\spec$ is in
\[
\bigO\left(\binom{\frac{m^2+m}{2}+n}{n}^3\right).
\]
\end{corollary}

\begin{proof}
Remark that the number of subsets $\iota \subset \{1, \ldots, m\}$, with $\sharp\iota=m-r$ is $\binom{m}{m-r}=\binom{m}{r}$,
and that {\sf SolveLMI} stops when $r$ reaches $\minrank$. Hence the first part follows by applying Proposition \ref{sym:prop:degree}.
Finally, remark that $p_0 \geq p_1 \geq \cdots \geq p_{r} \geq \cdots$ for all $m$, and hence, by Proposition \ref{sym:prop:degree}, the
degree of the rational parametrization whose solutions intersect $\spec$ is of the order of
$\binom{p_r+n}{n}^3 \leq \binom{p_0+n}{n}^3 = \binom{\frac{m^2+m}{2}+n}{n}^3$.
\end{proof}

In the column {\sf deg} of Table \ref{degrees+bounds} we report the degrees
of the rational parametrization $q_\iota$ returned by {\sf LowRankSymRec} at Step \ref{sym:step:lowrec:2a},
compared with its bound $\theta(m,n,r)$ computed in Proposition \ref{sym:prop:degree}.
For this table, the input are randomly generated symmetric pencils with rational coefficients.
When the algorithm does not compute critical points (that is, when the Lagrange system
generates the empty set) we put ${\sf deg}=0$. 

We recall that the routine {\sf LowRankSymRec} computes points in components of the real algebraic set $\calD_r \cap \RR^n$
not meeting the subset $\calD_{r-1} \cap \RR^n$, hence of the expected rank $r$. Moreover, we recall
that {\sf LowRankSym} calls recursively its subroutine {\sf LowRankSymRec}, eliminating at each call
the first variable. Hence, the total number of critical points computed by {\sf LowRankSym} for a given
expected rank $r$ is obtained by summing up the integer in column {\sf deg} for every
admissible value of $n$.
We remark here that both the degree and the bound are constant and equal to 0 if $n$ is large enough.
Hence, the previous sum is constant for large values of $n$. Similar behaviors appear, for example, when
computing the Euclidean Distance degree (EDdegree) of determinantal varieties, as in \cite{EDdegree}
or \cite{OtSpaStu}. In \cite[Table\,1]{OtSpaStu}, the authors report on the EDdegree of determinantal
hypersurfaces generated by linear matrices $A(x)=A_0+x_1A_1+\cdots+x_nA_n$: for generic weights in the distance
function, and when the codimension of the
vector space generated by $A_1,\ldots,A_n$ is small (for us, when $n$ is big, since matrices $A_i$ are
randomly generated, hence independent for $n \leq \binom{m+1}{2}=\dim {\mathbb{S}}_m(\QQ)$) the EDdegree is constant.
Similar comparisons can be done with data in \cite[Example\,4]{OtSpaStu} and \cite[Corollary\,3.5]{OtSpaStu}.

\begin{table}[!ht]
\centering
\begin{tabular}{|l|rr|l|rr|}
\hline
$(m,r,n)$ & {\sf deg} & $\theta(m,n,r)$ & $(m,r,n)$ & {\sf deg} & $\theta(m,n,r)$ \\
\hline
\hline
$(3,2,2)$ & 6 & 9 & $(4,3,9)$ & 0 & 0  \\
$(3,2,3)$ & 4 & 16 & $(5,2,5)$ & 0 & 0 \\
$(3,2,4)$ & 0 & 15 & $(5,2,6)$ & 35 & 924 \\
$(3,2,5)$ & 0 & 6 & $(5,2,7)$ & 140 & 10296 \\
$(3,2,6)$ & 0 & 0 & $(5,3,3)$ & 20 & 84 \\
$(4,2,3)$ & 10 & 35 & $(5,3,4)$ & 90 & 882 \\
$(4,2,4)$ & 30 & 245 & $(5,4,2)$ & 20 & 30 \\
$(4,2,5)$ & 42 & 896 & $(5,4,3)$ & 40 & 120 \\
$(4,2,6)$ & 30 & 2100 & $(5,4,4)$ & 40 & 325 \\
$(4,2,7)$ & 10 & 3340 & $(5,4,5)$ & 16 & 606 \\
$(4,2,8)$ & 0 & 3619 & $(6,3,3)$ & 0 & 0 \\
$(4,2,9)$ & 0  & 2576 & $(6,3,4)$ & 0 & 0 \\
$(4,2,12)$ & 0  & 0 & $(6,3,5)$ & 0 & 0 \\
$(4,3,3)$ & 16 & 52 & $(6,3,6)$ & 112 & 5005 \\
$(4,3,4)$ & 8 & 95 & $(6,4,2)$ & 0 & 0 \\
$(4,3,7)$ & 0 & 20 & $(6,4,3)$ & 35 & 165 \\
$(4,3,8)$ & 0 & 0 &  $(6,5,3)$ & 80 & 230 \\
\hline
\end{tabular}
\caption{Degrees and bounds for rational parametrizations}
\label{degrees+bounds}
\end{table}

The values in column {\sf deg} of Table \ref{degrees+bounds} must also be compared
with the associated algebraic degree of semidefinite programming. Given integers $k,m,r$ with
$r \leq m-1$, Nie, Ranestad, Sturmfels and von Bothmer computed in \cite{stu,SDPformula} formulas for the algebraic
degree $\delta(k,m,r)$ of a generic semidefinite program associated to $m \times m$ $k-$variate linear matrices,
with expected rank $r$. Since the values in column {\sf deg} match exactly the corresponding
values in \cite[Table 2]{stu}, we conclude this section with the following expected result, which is
a work in progress.

\begin{conjecture} \label{conj} Let $A \in {\mathbb{S}}_m^{n+1}(\QQ)$ be the
  input of {\sf SolveLMI}, and suppose that
  $\spec = \{ x \in \RR^n \mymid A(x) \succeq 0\}$ is not empty. Let
  $\delta(k,m,r)$ be the algebraic degree of a generic semidefinite
  program with parameters $k,m,r$ as in \cite{stu,SDPformula}. If
  $\sfH$ holds, then the sum of the degrees of the rational
  parametrizations computed during {\sf SolveLMI} is
\[
\sum_{r=1}^{\minrank}\,\,\binom{m}{r} \,\sum_{k=p_r-r(m-r)}^{\min(n,p_r+r(m-r))}\,\,\delta(k,m,r),
\]
where $p_r=(m-r)(m+r+1)/2$.
\end{conjecture}

\subsection{The complexity of {\sf SolveLMI}}
\label{sec4:ssec2}

\subsubsection*{Complexity of some subroutines}
We first provide complexity estimates for subroutines {\sf SolveLinear}, {\sf CheckLMI}, {\sf Project},
{\sf Lift}, {\sf Image} and {\sf Union}.

{\it Complexity of {\sf SolveLinear}}. This subroutine computes, if it exists, a solution of the $A(x)=0$.
This can be essentially performed by Gaussian elimination. The complexity of solving $\binom{m+1}{2}$ linear
equations in $n$ variables is linear in $\binom{m+1}{2}$ and cubic in $n$.

{\it Complexity of {\sf CheckLMI}}. This subroutine can be performed as follows. Let $q \subset \ZZ[t]$ be the
rational parametrization in the input of {\sf CheckLMI}. The spectrahedron $\spec = \{x \in \RR^n \mymid A(x)
\succeq 0\}$ is the semi-algebraic set defined, {\it e.g.}, by sign conditions on the coefficients of the
characteristic polynomial
\[
p(s,x) = \det (A(x)+s\,\Id_m) = f_m(x)+f_{m-1}(x)s + \cdots + f_{1}(x)s^{m-1} + s^m.
\]
That is, $\spec = \{x \in \RR^n \mymid f_i(x) \geq 0, \fall i=1, \ldots, m\}$.
We make the substitution $x_i \leftarrow q_i(t)/q_0(t)$ in $A(x)$ and
compute the coefficients of $p(s,x(t))$, that are rational functions of the
variable $t$. Hence {\sf CheckLMI} boils down to deciding on the sign of $m$
univariate rational functions (that is, of $2m$ univariate polynomials) over
the finite set defined by $q_{n+1}(t)=0$.
We deduce that the complexity of {\sf CheckLMI} is polynomial in $m$ and on
the degree of $q_{n+1}$ (that is, on the degree of $q$) see \cite[Ch.\,13]{BaPoRo06}.

{\it Complexity of {\sf Project, Lift, Image} and {\sf Union}} Estimates for the arithmetic complexities of
these routines are given in \cite[Ch.\,10]{SaSc13}. In particular, if $\theta=\theta(m,n,r)$ is the bound
computed in Proposition \ref{sym:prop:degree}, and $\tilde{n}=n+r(m-r)+p_r$, then:
\begin{itemize}
\item
By \cite[Lemma\,10.1.5]{SaSc13},
  {\sf Project} is performed within $\softO(\tilde{n}^2\theta^2)$ operations;
\item
By \cite[Lemma\,10.1.6]{SaSc13},
  {\sf Lift} is performed within $\softO(\tilde{n}\theta^2)$ operations;
\item
By \cite[Lemma\,10.1.1]{SaSc13},
  {\sf Image} is performed within $\softO(\tilde{n}^2\theta+\tilde{n}^3)$ operations;
\item
By \cite[Lemma\,10.1.3]{SaSc13},
  {\sf Union} is performed within $\softO(\tilde{n}\theta^2)$ operations.
\end{itemize}

\subsubsection*{Complexity of {\sf LowRankSym} and {\sf SolveLMI}}
The \new{complexity of {\sf LowRankSym} boils essentially down to the complexity of {\sf LowRankSymRec},
that is the complexity of {\sf RatPar}. This is performed with the symbolic homotopy algorithm
\cite{jeronimo2009deformation}: we bound its complexity in this section. We just remark that computing the
dimension of $\calD_r$ at Step \ref{sym:step:low:1} of {\sf LowRankSym} and performing the control
routine {\sf IsReg} can be done by combining the Jacobian criterion and Gr\"obner bases computations. 
Our complexity analysis omits this step.}

We recall that the simplified Lagrange system defined in the proof of Proposition \ref{sym:prop:degree}
contains: $p_r=(m-r)(m+r+1)/2$ polynomials of multidegree bounded by $(1,1,0)$; $n-1$ polynomials of
multidegree bounded by $(0,1,1)$; $r(m-r)$ polynomials of multidegree bounded by $(1,0,1)$.
Let us denote by $\ell$ this system. We denote by
\begin{align*}
\Delta_{xy} & = \{ 1, x_i, y_j, x_iy_j \mymid i=1,\ldots, n, j=1, \ldots, r(m-r) \} \\
\Delta_{yz} & = \{ 1, y_j, z_k, y_jz_k \mymid j=1,\ldots, r(m-r), k=2 \ldots, p_r \} \\
\Delta_{xz} & = \{ 1, x_i, z_k, x_iz_k \mymid i=1,\ldots, n, k=2, \ldots, p_r \}
\end{align*}
the supports of the aforementioned three groups of polynomials.
Let $\widetilde{\ell} \subset \QQ[x,y,z]$ be a polynomial system such that:
\begin{itemize}
\item the length of $\tilde{\ell}$ equals that of $\ell$;
\item for $i=1, \ldots, n-1+m^2-r^2$, the support of $\tilde{\ell}_i$ equals that of $\ell_i$;
\item the solutions of $\tilde{\ell}$ are known.
\end{itemize}
Remark that $\widetilde{\ell}$ can be easily built by considering suitable products of linear forms.
We build the homotopy
\begin{equation} \label{homotopy} 
\tau \ell + (1-\tau) \tilde{\ell} \subset \QQ[x,y,z,\tau],
\end{equation}
where $\tau$ is a new variable. The system \eqref{homotopy} defines a curve.
From \cite[Proposition 6.1]{jeronimo2009deformation}, if the solutions of $\tilde{\ell}$ are known, one can compute a rational
parametrization of $\zeroset{t \ell + (1-t) \tilde{\ell}}$ within $\bigO( (\tilde{n}^2N\log Q + \tilde{n}^{\omega+1}) e e')$
arithmetic operations over $\QQ$, where: $\tilde{n}$ is the number of variables in $\ell$; $N= p_r \sharp \Delta_{xy}+ (n-1) \sharp
\Delta_{yz}+r(m-r) \sharp \Delta_{xz}$; $Q = \max\{\Vert q \Vert : q \in \Delta_{xy} \cup \Delta_{yz} \cup \Delta_{xz}\}$; $e$ is the 
number of isolated solutions of $\ell$; $e'$ is the degree of $\zeroset{t \ell + (1-t) \tilde{\ell}}$; $\omega$ is the exponent of
matrix multiplication.

The technical lemma below gives a bound on the degree of $\zeroset{t \ell + (1-t) \tilde{\ell}}$.

\begin{lemma} \label{mult:bounds:homotopy:sym}
Let $\calG_{m,n,r}$ and $\theta(m,n,r)$ be the set and the bound defined in Proposition \ref{sym:prop:degree}, and suppose that
$\calG_{m,n,r} \neq \emptyset$. Let $e'$ be the degree of $\zeroset{t \ell + (1-t) \tilde{\ell}}$. Then
$e' \in \bigO\left((n+p_r+r(m-r))\,\min\{n,p_r\}\,\theta(m,n,r)\right).$
\end{lemma}

\begin{proof}
Similarly to Proposition \ref{sym:prop:degree}, we exploit the multilinear structure of $t \ell + (1-t) \tilde{\ell}$, to compute
a bound on $e'$. The system is bilinear with respect to $x,y,z,\tau$. We recall $\sharp x = n, \sharp y = r(m-r), \sharp z =
p_r-1, \sharp \tau = 1$, with $p_r = (m-r)(m+r+1)/2$. By \cite[Ch.\,11]{SaSc13}, $e'$ is bounded by the sum of the coefficients of
$
q=(s_x+s_y+s_\tau)^{p_r}(s_y+s_z+s_\tau)^{n-1}(s_x+s_z+s_\tau)^{r(m-r)}
$
modulo $I=\langle s_x^{n+1}, s_y^{r(m-r)+1}, s_z^{p_r}, s_\tau^2\rangle \subset
\ZZ[s_x,s_y,s_z,s_\tau]$. We see that $q = q_1 + s_\tau(q_2+q_3+q_4) + g$ with $s_\tau^2$ that divides
$g$ and
\begin{align*}
q_1 &= (s_x+s_y)^{p_r}(s_y+s_z)^{n-1}(s_x+s_z)^{r(m-r)} \\
q_2 &= p_r (s_x+s_y)^{p_r-1}(s_y+s_z)^{n-1}(s_x+s_z)^{r(m-r)} \\
q_3 &= (n-1)(s_x+s_y)^{p_r}(s_y+s_z)^{n-2}(s_x+s_z)^{r(m-r)}\\
q_4 &= r(m-r)(s_x+s_y)^{p_r}(s_y+s_z)^{n-1}(s_x+s_z)^{r(m-r)-1}.
\end{align*}
Hence $q \equiv q_1+s_\tau(q_2+q_3+q_4) \mod I$, and the bound is given by the sum of the contributions of $q_1, q_2, q_3$
and $q_4$ (multiplying $q_2,q_3,q_4$ by $s_\tau$ does not change the sum of the coefficients modulo $I$).
The contribution of $q_1$ is the sum of the coefficients of its class modulo $I'=\langle s_x^{n+1}, s_y^{r(m-r)+1}, s_z^{p_r} \rangle.$
This has been computed in Proposition \ref{sym:prop:degree}, and coincides with $\theta(m,n,r)$.
We compute the contribution of $q_2$.
Let $q_2=p_r \tilde{q}_2$. It is sufficient to compute the sum of the coefficients of $\tilde{q}_2$ modulo $I'$ (defined above), multiplied
by $p_r$. Since $\deg\,\tilde{q}_2 = n-2+p_r+r(m-r)$, and since the maximal powers admissible modulo $I'$ are
$s_x^{n}, s_y^{r(m-r)},$ and $s_z^{p_r-1}$, three configurations are possible.
\begin{itemize}
\item[(A)] The coefficient of $s_x^{n-1}s_y^{r(m-r)}s_z^{p_r-1}$ in $\tilde{q}_2$, that is
  \[
  \Sigma_A=\sum_{k}\binom{p_r-1}{n-1-k}\binom{n-1}{k-1+p_r-r(m-r)}\binom{r(m-r)}{k};
  \]
\item[(B)] the coefficient of $s_x^{n}s_y^{r(m-r)-1}s_z^{p_r-1}$ in $\tilde{q}_2$, that is
  \[
  \Sigma_B=\sum_{k}\binom{p_r-1}{n-k}\binom{n-1}{k-1+p_r-r(m-r)}\binom{r(m-r)}{k};
  \]
\item[(C)] the coefficient of $s_x^{n}s_y^{r(m-r)}s_z^{p_r-2}$ in $\tilde{q}_2$, that is
  \[
  \Sigma_C=\sum_{k}\binom{p_r-1}{n-k}\binom{n-1}{k-2+p_r-r(m-r)}\binom{r(m-r)}{k}.
  \]
\end{itemize}
Hence we need to bound the expression $p_r(\Sigma_A+\Sigma_B+\Sigma_C)$. One can easily check that
$\Sigma_A \leq \theta(m,n,r)$ and $\Sigma_B \leq \theta(m,n,r)$, while the same inequality is false
for $\Sigma_C$. However, we claim that $\Sigma_C \leq (1+\min\{n,p_r\})\,\theta(m,n,r)$ and
hence that the contribution of $q_2$ is $p_r(\Sigma_A+\Sigma_B+\Sigma_C) \in \bigO\left(p_r\,\min\{n,p_r\}\,\theta(m,n,r)\right)$.
We prove now this claim. Let
\begin{align*}
\chi_1 &= \max\{0,n-p_r\} \qquad \,\,\,\,\,\,\,\,\,\, \chi_2 = \min\{n-p_r+r(m-r),r(m-r)\} \\
\alpha_1 &= \max\{0,n-p_r+1\} \qquad \alpha_2 = \min\{n-p_r+r(m-r)+1,r(m-r)\}
\end{align*}
so that the sum in $\theta(m,n,r)$ runs over $\chi_1 \leq k \leq \chi_2$ and that in $\Sigma_C$ over $\alpha_1
\leq k \leq \alpha_2$. Remark that $\chi_1 \leq \alpha_1$ and $\chi_2 \leq \alpha_2$. 
Denote by $\varphi(k)$ the $k-$th term in the sum defining $\Sigma_C$, and
by $\gamma(k)$ the $k-$th term in the sum defining $\theta(m,n,r)$.
Then for all indices $k$, admissible both for $\theta(m,n,r)$ and $\Sigma_C$,
that is for $\alpha_1 \leq k \leq \chi_2$, one gets, by basic properties
of binomial coefficients, that
$\varphi(k) = \Psi(k)\,\gamma(k)$, with $\Psi(k) = \frac{k-1+p_r-r(m-r)}{n-k-p_r+r(m-r)-1}$.
When $k$ runs over all admissible indices, $\Psi(k)$ is non-decreasing monotone, and its maximum is $\Psi(\chi_2)$
and is bounded by $\min\{n,p_r\}$. We deduce the claimed inequality $\Sigma_C \leq (1+\min\{n,p_r\})\,\theta(m,n,r)$,
since if $\chi_2 < \alpha_2$ then $\chi_2 = \alpha_2-1$ and $\varphi(\alpha_2)$ is bounded above by $\theta(m,n,r)$.

{\it Contributions of $q_3$ and $q_4$.}
As for $q_2$, we deduce that the contribution of $q_3$ is in $\bigO\left(n\,\min\{n,p_r\}\,\theta(m,n,r)\right)$
and that of $q_4$ is in $\bigO\left(r(m-r)\,\min\{n,p_r\}\,\theta(m,n,r)\right)$.
\end{proof}

We use this degree estimate to conclude our complexity analysis of {\sf LowRankSym}.

\begin{proposition} \label{prop:comp:lowranksymrec}
Let $A \in {\mathbb{S}}_{m}^{n+1}(\QQ)$ be the input of {\sf SolveLMI} and $0 \leq r \leq m-1$.
Let $\theta(m,n,r)$ be the bound defined in Proposition \ref{sym:prop:degree}. 
Let $p_r = (m-r)(m+r+1)/2$. Then {\sf RatPar} returns a r.p. within
$
\softO\left(\binom{m}{r}\,(n+p_r+r(m-r))^{7}\,\theta(m,n,r)^2\right)
$
arithmetic operations over $\QQ$.
\end{proposition}

\begin{proof}
Let $\ell$ be the simplified Lagrange system as in the proof of Proposition \ref{sym:prop:degree}.
We consider the bound on the degree of the homotopy curve given by Lemma \ref{mult:bounds:homotopy:sym}.
We deduce the claimed complexity result by applying \cite[Proposition 6.1]{jeronimo2009deformation},
and by recalling that there are $\binom{m}{r}$ many subsets of $\{1,\ldots,m\}$ of cardinality $m-r$.
\end{proof}

We straightforwardly deduce the following complexity estimate for {\sf SolveLMI}.

\begin{theorem}[Complexity of {\sf SolveLMI}] \label{theo:compl:symm}
Let $A \in {\mathbb{S}}_m^{n+1}(\QQ)$ be the input symmetric pencil and suppose that $\sfH$ holds.
Then the number of arithmetic operations performed by {\sf SolveLMI} are in
\[
\softO\left( n\,\binom{\frac{m^2+m}{2}+n}{n}^6\,\sum_{r \leq m-1} \binom{m}{r}\,(n+(m-r)(m+3r))^{7}\right).
\]
\end{theorem}

\begin{proof}
From Proposition \ref{prop:comp:lowranksymrec}, we deduce that {\sf LowRankSymRec} runs essentially
within $\softO(\binom{m}{r}\,(n+p_r+r(m-r))^{7}\,\theta(m,n,r)^2)$ arithmetic operations.
The inequality $\theta(m,n,r) \leq \binom{n+p_r}{n}^3$ is proved in Proposition \ref{sym:prop:degree}.
Moreover, there are at most $n$ recursive calls of {\sf LowRankSymRec} in {\sf LowRankSym},
and {\sf SolveLMI} stops at most when $r$ reaches $m-1$. Finally, the cost of subroutines {\sf SolveLinear},
{\sf CheckLMI}, {\sf Project, Lift, Image} and {\sf Union} is negligible.
We deduce that the complexity of {\sf SolveLMI} is in
$
\softO\left( n\,\sum_{r \leq m-1} \binom{m}{r}\,(n+p_r+r(m-r))^{7}\,\binom{p_r+n}{n}^6\right).
$
Since $p_r \leq p_0 = \frac{m^2+m}{2}$ and $p_r+r(m-r) \leq (m-r)(m+3r)$, we conclude.
\end{proof}


\section{Experiments} \label{sec5}
{\sf SolveLMI} is implemented in a {\sc maple} function, and it is part of a library called {\sc spectra} (Semidefinite
Programming solved Exactly with Computational Tools of Real Algebra), released in September 2015. It collects efficient
and exact algorithms solving a large class of problems in real algebraic geometry and semidefinite optimization.

We present in this section the results of our experiments. These have been performed on a machine with the following
characteristics: Intel(R) Xeon(R) CPU E7540@2.00GHz with 256 Gb of RAM. We rely on the {\sc maple} implementation of the
software {\sc FGb} \cite{faugere2010fgb}, for fast computation of Gr\"obner bases. To compute the rational parametrizations
we use the implementation in {\sc maple} of the change-of-ordering algorithm FGLM \cite{fglm} and of its improved versions
\cite{FM13,newfglm}.

\subsection{Generic symmetric pencils}
\label{sec5:ssec1}

We implemented the function {\sf LowRankSym} and tested the running time of the implementation with input generic symmetric
linear matrices. We recall that the algorithm {\sf SolveLMI} amounts to iterating {\sf LowRankSym} by increasing the
expected rank $r$. \new{By generic data we mean that} a natural number $N \in \NN$ \new{large enough is fixed}, and numerators
and denominators \new{of the entries of $A_\ell, \ell=0, \ldots, n$ are} uniformly \new{generated} in the interval $[-N,N]$.
We report in Table \ref{tab:dense:sym} the timings and the degrees of output rational parametrizations.

{\small
\begin{table}[!ht]
\centering
{\footnotesize
\begin{tabular}{|l|rrrr||l|rrrr|}
\hline
$(m,r,n)$ & {\sf PPC} & {\sf LRS} & {\sf totaldeg} & {\sf deg} & $(m,r,n)$ & {\sf PPC} & {\sf LRS} & {\sf totaldeg} & {\sf deg} \\
\hline
\hline
$(3,2,2)$ & 0.2 & 8   & 9   & 6  &       $(4,3,9)$ & $\infty$ & 28 & 40 & 0 \\
$(3,2,3)$ & 0.3 & 11  & 13  & 4  &       $(4,3,10)$ & $\infty$ & 29 & 40 & 0 \\
$(3,2,4)$ & 0.9 & 13  & 13  & 0     &    $(4,3,11)$ & $\infty$ & 30 & 40 & 0 \\
$(3,2,5)$ & 5.1 & 14  & 13  & 0   &      $(5,2,2)$ & 0.6 & 0 & 0 & 0 \\
$(3,2,6)$ & 15.5 & 15 & 13  & 0   &      $(5,2,3)$      & 0.9 & 0 & 0 & 0 \\
$(3,2,7)$ & 31 & 16   & 13  & 0   &      $(5,2,4)$        & 1 & 1  & 0 & 0 \\
$(3,2,8)$ & 109 & 17  & 13  & 0   &      $(5,2,5)$     & 1.6 & 1 & 0 & 0 \\
$(3,2,9)$ & 230 & 18  & 13  & 0  &       $(5,2,7)$     & $\infty$ & 25856  & 175 & 140   \\
$(4,2,2)$ & 0.2 & 0   & 0   & 0   &      $(5,3,2)$    & 0.4 & 1  & 0 & 0 \\
$(4,2,3)$ & 0.3 & 2   & 10  & 10   &     $(5,3,3)$    & 0.5 & 3  & 20 & 20 \\
$(4,2,4)$ & 2.2 & 9   & 40  & 30 &       ${(5,3,4)}$    & $\infty$ & 1592 & 110 & 90   \\ 
$(4,2,5)$ & 12.2 & 29 & 82  & 42 &       $(5,3,5)$  &  $\infty$ & 16809 & 317 & 207  \\
$(4,2,6)$ & $\infty$ & 71 & 112 & 30 &   $(5,4,2)$   & 0.5 & 7  & 25 & 20 \\
$(4,2,7)$ & $\infty$ & 103 & 122 & 10  & $(5,4,3)$    & 10 & 42  & 65 & 40  \\
$(4,2,8)$ & $\infty$ & 106 & 122 & 0 &   $(5,4,4)$   & $\infty$ & 42  & 105 & 40 \\
$(4,2,9)$ & $\infty$ & 106 & 122 & 0 &   $(5,4,5)$   & $\infty$ & 858 & 121 & 16    \\
$(4,3,3)$ & 1 & 10  & 32 & 16 &          $(6,3,3)$    & 4 & 0  & 0  & 0  \\
$(4,3,4)$ & 590 & 21  & 40 & 8 &         $(6,3,4)$    & 140 & 1 & 0 & 0   \\
$(4,3,5)$ & $\infty$ & 22 & 40 & 0  &    $(6,3,5)$    & $\infty$ & 2  & 0 & 0  \\ 
$(4,3,6)$ & $\infty$ & 24 & 40 & 0   &   $(6,3,6)$    & $\infty$ & 704 & 112  & 112   \\ 
$(4,3,7)$ & $\infty$ & 26 & 40 & 0 &     $(6,4,2)$    & 0.6 & 1 & 0 & 0     \\
$(4,3,8)$ & $\infty$ & 27 & 40 & 0   &   ${(6,5,3)}$   & $\infty$ & 591  & 116 & 80 \\
\hline
\end{tabular}
}
\caption{Timings and degrees for dense symmetric linear matrices}
\label{tab:dense:sym}
\end{table}
}

\new{We recall that} $m$ is the size of the input matrix, $n$ is the
number of variables and $r$ is the expected maximum rank (that is, the
index of the algebraic set $\calD_r$). We compare our timings
(reported in {\sf LRS}) with those of the function {\sf
  PointsPerComponents} (column {\sf PPC}) of the library {\sc raglib}
developed by the third author \cite{raglib}. The input of {\sf
  PointsPerComponents} are the $(r+1) \times (r+1)$ minors of the
linear matrix, and the output is a rational parametrization of a
finite set meeting each connected component of $\calD_r \cap
\RR^n$.
\new{We do not consider the time needed to compute all the minors of
  $A(x)$ in {\sf PPC}.}  The symbol $\infty$ means that we did not
succeed in computing the parametrizations after 48 hours. Column {\sf
  deg} contains the degree of the parametrization returned by {\sf
  LowRankSymRec} at Step \ref{sym:step:lowrec:2a}, or $0$ if the empty
list is returned. Column {\sf totaldeg} contains the sum of the values
in {\sf deg} for $k$ varying between $1$ and $n$.  For example, for
$m=4,r=2$, for $n \leq 2$ and $n\geq 8$ the algorithm does not compute
critical points, while it computes rational parametrizations of degree
respectively $10,30,42,30,10$ for $n=3,4,5,6,7$; the number 82 in {\sf
  totaldeg} for $(m,n,r)=(4,2,5)$ is obtained as the sum $10+30+42$ of
the integers in {\sf deg} for $m=4,r=2$ and $n=3,4,5$.  We remark
that, as for Table \ref{degrees+bounds}, the value in {\sf deg} for a
given triple $m,n,r$ coincides with the algebraic degree of
semidefinite programming, that is $\delta(n,m,r)$ as defined in
\cite{stu}.

Our algorithm allows to tackle examples that are out of reach for {\sc
  raglib} and that, most of the time, the growth in terms of running
time is controlled when $m,r$ are fixed. This shows that our dedicated
algorithm leads to practical remarkable improvements: indeed, for
example, $4 \times 4$ linear matrices of expected rank $2$ are treated
in a few minutes, up to linear sections of dimension $9$; we are also
able to sample hypersurfaces in $\RR^5$ defined by the determinant of
$5 \times 5$ symmetric linear matrices; finally, symmetric linear
matrices of size up to $6$ with many rank defects are shown to be
tractable by our approach.
We observe that most of the time is spent to
compute a Gr\"obner basis of the Lagrange systems, and for this we use new fast algorithms
for the change of monomial orderings \cite{newfglm}
.

\subsection{Scheiderer's spectrahedron}
\label{sec5:ssec2}
We consider the symmetric pencil
{\footnotesize \[
A(x)=
\left(
\begin{array}{cccccc}
1 & 0 & x_1 & 0 & -{3}/{2}-x_2 & x_3 \\
0 & -2x_1 & {1}/{2} & x_{2} & -2-x_4 & -x_5 \\
x_1 & {1}/{2} & 1 & x_4 & 0 & x_6 \\
0 & x_2 & x_4 & -2x_3+2 & x_5 & {1}/{2} \\
-{3}/{2}-x_2 & -2-x_4 & 0 & x_5 & -2x_6 & {1}/{2} \\
x_3 & -x_5 & x_6 & {1}/{2} & {1}/{2} & 1
\end{array}
\right).
\]}
which is the {\it Gram matrix} of the trivariate polynomial
\[
f(u_1,u_2,u_3) = u_1^4+u_1u_2^3+u_2^4-3u_1^2u_2u_3-4u_1u_2^2u_3+2u_1^2u_3^2+u_1u_3^3+u_2u_3^3+u_3^4.
\]
In other words, $f = v^T A(x) v$ for all $x \in \RR^6$,
where $v=(u_1^2,u_1u_2,u_2^2,u_1u_3,u_2u_3,u_3^2)$ is the monomial basis of
the vector space of homogeneous polynomials of degree 2 in $u_1,u_2,u_3$.
Since $f$ is globally nonnegative, by Hilbert's theorem \cite{Hilbert3} it is a sum of at most three
squares in $\RR[u_1,u_2,u_3]$, namely there exist $f_1,f_2,f_3 \in \RR[u_1,u_2,u_3]$
such that $f=f_1^2+f_2^2+f_3^2$. Moreover, the spectrahedron
$\spec = \{x \in \RR^6 \mymid A(x) \succeq 0\}$ parametrizes all the
sum-of-squares decompositions of $f$ (and it is a particular example of
a {\it Gram spectrahedron}). Hence $\spec$ must contain a matrix of rank
at most $3$.

Scheiderer
proved in \cite{scheid} that $f$ does not admit a sum-of-squares
decomposition in the ring $\QQ[u_1,u_2,u_3]$, that is, the summands
$f_1,f_2,f_3$ cannot
be chosen to have rational coefficients, answering a
question of Sturmfels. We deduce that $\spec$ does not contain points
with rational coordinates. In particular, it is not full-dimensional (its
affine hull has dimension $\leq 5$) by straightforward density arguments.

We first easily check that $\calD_0 \cap \RR^6 = \calD_1 \cap \RR^6 = \emptyset$).
Further, for $r=2$, the algorithm returns the following rational parametrization
of $\calD_2 \cap \RR^6$:
{\small\[
\left( \frac{3+16t}{-8+24t^2}, \frac{8-24t^2}{-8+24t^2}, \frac{8+6t+8t^2}{-8+24t^2},
\frac{16+6t-16t^2}{-8+24t^2}, \frac{-3-16t}{-8+24t^2}, \frac{3+16t}{-8+24t^2} \right).
\]}
where $t$ satisfies $8t^3-8t-1 = 0$.
The set $\calD_2$ is, indeed, of dimension $0$, degree $3$, and it contains
only real points. Remark that the technical assumption $\sfP_2$ is not satisfied here,
since the expected dimension of $\calD_2$ is $-1$.
Conversely, the regularity assumptions on the incidence varieties are
satisfied.
By applying {\sf CheckLMI} one gets that two of the three
points lie on $\spec$, that is those with the following floating
point approximation up to 9 certified digits:
{\small
\[
\left(
\begin{array}{r}
-0.930402926 \\
-1.000000000 \\
 0.731299211 \\
-0.268700788 \\
 0.930402926 \\
-0.930402926 
\end{array}
\right)
\qquad \text{and} \qquad
\left(
\begin{array}{r}
-0.127050844 \\
-1.000000000 \\
-0.967716166 \\
-1.967716166 \\
 0.127050844 \\
-0.127050844
\end{array}
\right).
\]
}
These correspond to the two distinct decompositions of $f$ as a sum of 2 squares.
An approximation of such representations can be computed
by factorizing the matrix $A(x(t^*))= V^TV$ where $t^*$ is the corresponding root
of $8t^3-8t-1$ and $V \in \MM_{2,6}(\RR)$ is full rank. The corresponding
decomposition is $f=v^TV^TVv=\vert\vert V v\vert\vert^2$.
At the third point of $\calD_2 \cap \RR^6$
the matrix $A(x)$ is indefinite, so it is not a valid Gram matrix.

To conclude, {\sf SolveLMI} allows to design a computer-aided proof of Scheiderer's results.
This example is interesting since the interior of $\spec$ is empty and, typically, this can lead
to numerical problems when using interior-point algorithms. 

\section{Conclusion}
\label{sec6}
We have presented a \new{probabilistic} exact algorithm that computes
an algebraic representation of at least one feasible point of a LMI
$A(x) \succeq 0$, or that detects emptiness of
$\spec = \{x \in \RR^n \mymid A(x) \succeq 0\}$. \new{The algorithm
  works under assumptions which are proved to be generically
  satisfied.} \new{When these assumptions are not satisfied, the
  algorithm may return a wrong answer or raises an error (when the
  dimension of some Lagrange system is not $0$).} The main strategy is
to reduce the input problem to a sequence of real root finding
problems for the loci of rank defects of $A(x)$: if $\spec$ is not
empty, we have shown that computing sample points on determinantal
varieties is sufficient to sample $\spec$, and that it can be done
efficiently. Indeed, the arithmetic complexity is essentially
quadratic on a multilinear B\'ezout bound on the output degree.

This is, to our knowledge, the first exact computer algebra algorithm
tailored to linear matrix inequalities.  We conjecture that our
algorithm is optimal since the degree of the output parametrization
matches the algebraic degree of a generic semidefinite program, with
expected rank equal to the minimal achievable rank on $\spec$. Since
deciding the emptiness of $\spec$ is a particular instance of
computing the minimizer of a linear function over this set (namely, of
a constant), our algorithm is able to compute minimal-rank solutions
of special semidefinite programs, which is, in general, a hard
computational task. Indeed, numerical interior-point algorithms
typically return approximations of feasible matrices with maximal rank
among the solutions (those lying in the relative interior of the
optimal face). Moreover, the example of Scheiderer's spectrahedron
shows that we can also tackle degenerate situations with no interior
point which are typically numerically troublesome.

To conclude, as highlighted by the discussions in Section \ref{sec5},
our viewpoint includes an effective aspect, by which it is essential
to translate into practice the complexity results that have been
obtained. This is the objective of our {\sc maple} library {\sc
  spectra}. It must be understood as a starting point towards a
systematic exact computer algebra approach to semidefinite programming
and related questions.

\end{document}